\documentclass[12pt, a4paper]{article}
\usepackage{alnumsec}
\usepackage{amsmath}
\usepackage{amssymb,amscd}
\usepackage{times}
\usepackage{pifont}
\usepackage{csquotes}
\usepackage{longtable}
\usepackage{array}
\usepackage{tabularx}
\usepackage{psfrag, graphicx}
\usepackage[percent]{overpic}
\usepackage{url}
\makeatletter \@addtoreset{equation}{section}

\makeatother

\newtheorem{theorem}{Theorem}[section]
\newtheorem{definition}[theorem]{Definition}

\newtheorem{lemma}[theorem]{Lemma}
\newtheorem{corollary}[theorem]{Corollary}
\newtheorem{remark}[theorem]{Remark}

\newtheorem{algorithm}{Algorithm}
\makeatletter
\newenvironment{subtheorem}[1]{%
  \def\subtheoremcounter{#1}%
  \refstepcounter{#1}%
  \protected@edef\theparentnumber{\csname the#1\endcsname}%
  \setcounter{parentnumber}{\value{#1}}%
  \setcounter{#1}{0}%
  \expandafter\def\csname the#1\endcsname{\theparentnumber\Alph{#1}}%
  \ignorespaces
}{%
  \setcounter{\subtheoremcounter}{\value{parentnumber}}%
  \ignorespacesafterend
}
\makeatother
\newcounter{parentnumber}

\newenvironment{proof}{\par \vspace{0.3cm} \noindent{\sc Proof:} \ignorespaces}%
{\nolinebreak\hfill $\square$\par \medskip}

\title{Approximation of Lyapunov Functions from Noisy Data}

\author{P. Giesl\thanks{Department of Mathematics, University of Sussex, Falmer, BN1 9QH, UK ({p.a.giesl@sussex.ac.uk})},
\and B. Hamzi\thanks{ Department of Mathematics, Ko\c{c} University, Istanbul, Turkey   \&   Department of Mathematics, Alfaisal University, Riyadh,  Saudi Arabia  ({bhamzi@alfaisal.edu})},
\and M. Rasmussen\thanks{Department of Mathematics, Imperial College London, SW7 2AZ, UK ({m.rasmussen@imperial.ac.uk})},
\and K. N. Webster\thanks{Department of Mathematics, Imperial College London, SW7 2AZ, UK %(\email{kevin.webster@imperial.ac.uk})
\&
Potsdam Institute for Climate Impact Research, PO Box 60 12 03, 14412 Potsdam, Germany ({kevin.webster@pik-potsdam.de}).}
}

\begin{document}
\maketitle
%\slugger{sinum}{xxxx}{xx}{x}{x--x}%slugger should be set to mms, siap, sicomp, sicon, sidma, sima, simax, sinum, siopt, sisc, or sirev

\begin{abstract}
Methods have previously been developed for the approximation of Lyapunov functions using radial basis functions. However these methods assume that the evolution equations are known. We consider the problem of approximating a given Lyapunov function using radial basis functions where the evolution equations are not known, but we instead have sampled data which is contaminated with noise. We propose an algorithm in which we first approximate the underlying vector field, and use this approximation to then approximate the Lyapunov function. Our approach combines elements of machine learning/statistical learning theory with the existing theory of Lyapunov function approximation. Error estimates are provided for our algorithm.
\end{abstract}

\pagestyle{myheadings}
\thispagestyle{plain}
\markboth{P.~Giesl, B.~Hamzi, M.~Rasmussen \& K.N.~Webster}{Approximation of Lyapunov functions from noisy data}

\section{Introduction}			\label{sec:introduction}

Ordinary differential equations model large classes of applications such as planetary motion, chemical reactions, population dynamics or consumer behaviour. A breakthrough in the understanding of ordinary differential equations was initiated by Poincar\'e and Lyapunov in the late 19th century, who developed an approach that embraced the use of topological and geometrical techniques for the study of dynamical systems. 
A key component of this theory is Lyapunov functions, which can be used to determine the basin of attraction of an asymptotically stable equilibrium.

In general, it is not possible to find an explicit analytical expression for a Lyapunov function associated to a nonlinear
differential equation. Many methods have been proposed to numerically construct Lyapunov functions, see \cite{GieHaf15:b} for a recent review. These methods include the SOS (sums of squares) method, which constructs a polynomial Lyapunov function by semidefinite optimization \cite{prajna}.
Another method constructs
   a continuous piecewise affine (CPA) Lyapunov function using linear optimization \cite{Haf2007mon}. A further method is based on Zubov's equation and computes a solution of this partial differential equation  \cite{zubov2001camilli}.
   Lyapunov functions can also be constructed using
 set oriented
methods  \cite{book2002grune}.
The method that is also used in this paper is based on approximating the solution of a PDE using radial basis functions \cite{Gie07:a}.
All these methods to approximate Lyapunov functions rely on the knowledge of the right hand side of the
differential equation.

%Moreover, several methods have been developed to compute the domain of
%attraction of a dynamical system defined by a differential or difference
%equation. For example, there are methods based on estimating the basin
%of attraction based on the level sets of Lyapunov functions \cite{Gie07:a,GieHaf15}.
%%Other methods based on Zubov method (Zubov, Wirth
%%for extensions and other references therein).
%These methods also depend
%on the knowledge of the right hand side of the differential (resp.
%difference equation).

In this paper, we develop a method to approximate Lyapunov functions where the right hand side is unknown, but we have sampled data of the system, which is contaminated by noise.
We will first approximate the right hand side of the differential equation, and then use this approximation to approximate the Lyapunov function. Our approach combines and develops previous results from  statistical learning theory \cite{SmaZho04, SmaZho05, SmaZho07} together with existing methods using radial basis functions \cite{Gie07:a, GieWen07}, which use the framework of reproducing kernel Hilbert spaces (RKHS).

The use of RKHS spaces to approximate important quantities in dynamical systems has previously been exploited by Smale and Zhou to approximate a hyperbolic dynamical system  \cite{hyperbolic}. Bouvrie and Hamzi also use RKHS spaces to approximate some key quantities in control and random dynamical systems \cite{allerton, acc2012}.

%On the other hand, data-based modelling of nonlinear dynamical systems
%has been addressed by many authors. For example,  several methods have
%been developed in Time Series Analysis (e.g. \cite{kantz} for example) and System Identification (\cite{ljung}  for
%example).
%%Coifman et al. discuss data-based modelling of a stochastic
%%Langevin system \cite{coifman}. Archambeau et al. \cite{archambeau} proposed methods to
%%approximate SDEs from data.
%Smale and Zhou  use kernel methods to
%approximate a hyperbolic dynamical system \cite{hyperbolic}. Bouvrie and Hamzi use kernel
%methods for the approximation of nonlinear systems \cite{allerton, acc2012, siam2014,scl2015} in view of estimating some key quantities in control
%and random dynamical systems.

%Our goal in this paper is to perform a data-based analysis of nonlinear
%dynamical systems using tools from statistical learning theory by
%generalising  the existing method of approximation by radial basis functions
%\cite{Gie07:a,GieWen07} to the case where the right hand side is not known, but (noisy) measurements
%are given  at a discrete set of sampled locations.

% Our approach
%goes as follows: First, we derive error bounds for the approximation of
%a nonlinear differential equation.
%Then we apply the method in \cite{GieWen07} to this approximation.  We
%derive error bounds to show that if the equilibrium point is exponentially
%asympotically stable then our algorithm generates a Lyapunov function
%for the nonlinear system.

%%%%%%%%%%%%%%%%%%%%%%%%%%%%%%%%%%%%%%%%%%%%%%%%%%%%%%%%%

\section{Setting of the Problem and Main Result}				\label{sec:setting}

We consider ordinary differential equations of the form
%\begin{subequations}	
%\label{eqn:dynsys}		
%\begin{align}
%\dot{x} & = f^*_C(x)	\label{eqn:dynsyscont}\\
%\text{or}\quad x_{n+1} & = f^*_D(x_n) \label{eqn:dynsysdisc}
%\end{align}
%\end{subequations}
\begin{equation}
  \dot{x}  = f^*(x), 	\label{eqn:dynsys}
\end{equation}
where $f^*: \mathbb{R}^d \to \mathbb{R}^d$ is a smooth vector field and dot denotes differentiation with respect to time.
%We are interested in solutions , $t\ge 0$, of \eqref{eqn:dynsys}, where $x(0)=\psi$, with $\psi\in\mathbb{R}^d$.
We define the flow $\varphi_{f^*}:\mathbb{R}^d \times \mathbb{R}\rightarrow \mathbb{R}^d$ by $\varphi_{f^*}(\psi,t) := x(t)$, where $x(t)$ solves \eqref{eqn:dynsys} with $x(0)=\psi$.

We assume that \eqref{eqn:dynsys} has a fixed point $\overline{x}$ that is  exponentially asymptotically stable. Define the basin of attraction as $A(\overline{x}):=\{\psi\in\mathbb{R}^d\mid \lim_{t\rightarrow\infty}\varphi_{f^*}(\psi,t)=\overline{x}\}$. Note that $A(\overline{x})\ne\emptyset$ and $A(\overline{x})$ is open. Subsets of the basin of attraction can be determined by the use of Lyapunov functions, which are functions decreasing along solutions of \eqref{eqn:dynsys}. We consider two types of Lyapunov functions $V$ and $T$, as described in Theorems~\ref{thm:VLyapunovconverse} and \ref{thm:TLyapunovconverse} below. These Lyapunov functions satisfy
\begin{eqnarray*}
\langle \nabla V(x),f^*(x)\rangle_{\mathbb{R}^d} &=& -p(x), \qquad x\in A(\overline{x}),\\
\langle \nabla T(x),f^*(x)\rangle_{\mathbb{R}^d} &=& -\overline{c}, \qquad x\in A(\overline{x})\setminus\{\overline{x}\},
\end{eqnarray*}
where $p$ is a smooth function with $p(x)>0$ for $x\not=\bar x$ and $p(\bar x)=0$, and $\overline{c}$ is a positive constant. The scalar products on the left hand sides are called the orbital derivatives of $V$ and $T$ with respect to \eqref{eqn:dynsys}, which are the derivatives of $V$ and $T$ along solutions of \eqref{eqn:dynsys}. The orbital derivatives of $V$ and $T$ are negative, which implies that $V$ and $T$ are decreasing along solutions.

We assume that the function $f^*$ is unknown, but we have sampled data of the form $(x_i, y_i)$ in $X\times \mathbb{R}^d$, $i=1,\dots,m$, with $y_i = f^*(x_i) + \eta_{x_i}$. We assume that the one-dimensional random variables $\eta_{x_i}^k \in \mathbb{R}^d$, where $i=1,\dots,m$ and $k=1,\dots,d$, are independent random variables drawn from a probability distribution with zero mean and variance $(\sigma_{x_i}^k)^2$ bounded by $\sigma^2$. Here $X$ is a nonempty and compact subset of $\mathbb{R}^d$ with $C^1$ boundary.

In \S\ref{sec:algorithm} we provide an algorithm to approximately reconstruct the Lyapunov functions $V$ and $T$ by  functions $\hat{V}$  and $\hat{T}$. The following main theorem provides error estimates in a compact set $\mathcal{D} \subset A(\overline{x})\cap X$, which depend on the density of the data, measured by two key quantities: the fill distance of the data $h_\mathbf{x}$ (see Definition \ref{def:filldistance}) and the norm of the volume weights $\mathbf{w}$ corresponding to the Voronoi tessellation of the data (see Definition~\ref{def:Voronoi}).

\begin{theorem}\label{thm:mainresult}
 Consider \eqref{eqn:dynsys} such that $f^* \in C^{\nu_1}(\mathbb{R}^d,\mathbb{R}^d)$ with $\nu_1\ge (3d + 7)/2$ if $d$ is odd, or $\nu_1\ge (3d + 12)/2$ if $d$ is even. Let $\tau_1,\tau_2\in\mathbb{R}$ and $k_1, k_2\in\mathbb{N}$ be such that $\tau_1 = k_1 + (d+1)/2$ with $\lceil \tau_1 \rceil  = \nu_1$, and $k_2 = k_1 - (d+2)$ (if $d$ is odd) or $k_2 = k_1-(d+3)$ (if $d$ is even). Define $\tau_2:=k_2 + (d+1)/2$.

  Let $\Omega\subset A(\overline{x})$ be a compact set and $\mathcal{D}:= \Omega \setminus B_\varepsilon(\overline{x})\subset X$, with $\varepsilon>0$  small enough so that $\mathcal{D}\ne\emptyset$. For $h_\mathbf{x}$, $||\mathbf{w}||_{\mathbb{R}^m}$ and $h_\mathbf{q}$ sufficiently small, the following holds:
\begin{enumerate}
\item
For every $0<\delta<1$, the reconstruction $\hat{V}$ of the Lyapunov function $V$ defined in Theorem \ref{thm:VLyapunovconverse} satisfies the following estimate with probability $1-\delta$:
% LONG VERSION:
%\begin{eqnarray}
%\hspace{-1cm}\left| \left| \langle\nabla \hat{V},f^*\rangle_{\mathbb{R}^d} - \langle\nabla {V},f^*\rangle_{\mathbb{R}^d}    \right| \right|_{L^\infty(\mathcal{D})} & \le &
%C  ||V||_{W_2^{\tau_2}(\Omega_V)} . \left(h_\mathbf{q}^{k_2-\frac{1}{2}}
%+ \frac{||\mathbf{w}||_{\mathbb{R}^m}\sigma\kappa^2}{\lambda \sqrt{\delta}} \right. \nonumber\\
%& & + \lambda^{r-1}4^s ||K^{1}||^{1/2}_{C^{2s}}||L_{K^{1}}^{-r}f^*||_{\mathcal{L}^2_{\rho}} \kappa^2 h_\mathbf{x} \rho(X)\nonumber\\
%& & \left. + \kappa\lambda^{r-\frac{1}{2}}||L_{K^{1}}^{-r}f^*||_{L^\infty(X)}\right),				\label{eqn:mainresultV}
%\end{eqnarray}
% SHORTER, BUT STILL INCLUDING LAMBDA:	
\begin{eqnarray}
\left| \left| \langle\nabla \hat{V},f^*\rangle_{\mathbb{R}^d} - \langle\nabla {V},f^*\rangle_{\mathbb{R}^d}    \right| \right|_{L^\infty(\mathcal{D})}  &\le &
C  ||V||_{W_2^{\tau_2}(\Omega_V)}  \left(h_\mathbf{q}^{k_2-\frac{1}{2}}\right.\nonumber\\
&&\left.+ \frac{||\mathbf{w}||_{\mathbb{R}^m}}{\lambda \sqrt{\delta}} + \lambda^{r-\frac{3}{2}} h_\mathbf{x}  + \lambda^{r-\frac{1}{2}}\right),				\label{eqn:mainresultV}
\end{eqnarray}	
% SUBSTITUTE LAMBDA WITH OPTIMAL CHOICE:
%\begin{eqnarray}
%\left| \left| \langle\nabla \hat{V},f^*\rangle_{\mathbb{R}^d} - \langle\nabla {V},f^*\rangle_{\mathbb{R}^d}    \right| \right|_{L^\infty(\mathcal{D})}  &\le &
%C  ||V||_{W_2^{\tau_2}(\Omega_V)}  \left(h_\mathbf{q}^{k_2-\frac{1}{2}}\right.\nonumber\\
%&&\left.\left( \max\left\{\frac{||\mathbf{w}||_{\mathbb{R}^m}}{\sqrt{\delta}} , h_\mathbf{x} \right\}\right)^{\frac{2r-1}{2r+1}}\right),				\label{eqn:mainresultV}
%\end{eqnarray}	
where $\Omega_V\supset \mathcal{D}$  is a certain compact subset of $A(\overline{x})$, and $\frac{1}{2} < r \le 1$.
\item For every $0<\delta<1$, the reconstruction $\hat{T}$ of the Lyapunov function $T$ defined in Theorem \ref{thm:TLyapunovconverse} satisfies the following estimates with probability $1-\delta$:
%\begin{eqnarray}
%\hspace{-1cm}\left| \left| \langle\nabla \hat{T},f^*\rangle_{\mathbb{R}^d} - \langle\nabla {T},f^*\rangle_{\mathbb{R}^d}  \right| \right|_{L^\infty(\mathcal{D})} & \le &
%C  ||T||_{W_2^{\tau_2}(\Omega_T)} . \left(h_\mathbf{q}^{k_2-\frac{1}{2}}
%+ \frac{||\mathbf{w}||_{\mathbb{R}^m}\sigma\kappa^2}{\lambda \sqrt{\delta}} \right. \nonumber\\
%& &  + \lambda^{r-1}4^s ||K^{1}||^{1/2}_{C^{2s}}||L_{K^{1}}^{-r}f^*||_{\mathcal{L}^2_{\rho}} \kappa^2 h_\mathbf{x} \rho(X)\nonumber\\
%& & \left. + \kappa\lambda^{r-\frac{1}{2}}||L_{K^{1}}^{-r}f^*||_{L^\infty(X)}\right).				\label{eqn:mainresultT1}\\
%\hspace{-1cm}\left| \left| \hat{T} - T\right| \right|_{L^\infty(\Gamma)} & \le &C h_{\tilde{\mathbf{q}}}^{k+\frac{1}{2}}||T||_{W_2^{\tau_2}(\Omega_T)}
%\label{eqn:mainresultT2}
%\end{eqnarray}	
\begin{eqnarray}
\left| \left| \langle\nabla \hat{T},f^*\rangle_{\mathbb{R}^d} - \langle\nabla {T},f^*\rangle_{\mathbb{R}^d}  \right| \right|_{L^\infty(\mathcal{D})} & \le &
C  ||T||_{W_2^{\tau_2}(\Omega_T)} \left(h_\mathbf{q}^{k_2-\frac{1}{2}}\right.\nonumber\\
&&\left.+ \frac{||\mathbf{w}||_{\mathbb{R}^m}}{\lambda \sqrt{\delta}}+ \lambda^{r-\frac{3}{2}} h_\mathbf{x}  + \lambda^{r-\frac{1}{2}}\right),				\label{eqn:mainresultT1}\\
\left| \left| \hat{T} - T\right| \right|_{L^\infty(\Gamma)} & \le &C h_{\tilde{\mathbf{q}}}^{k_2+\frac{1}{2}}||T||_{W_2^{\tau_2}(\Omega_T)},
\label{eqn:mainresultT2}
\end{eqnarray}	
where $\Gamma$ is a non-characteristic hypersurface on which $T$ has defined values (see Definition \ref{def:noncharhyp}),  $\Omega_T\supset\mathcal{D}$ is a certain compact subset of $A(\overline{x})$, and $\frac{1}{2} < r \le 1$.\\
\end{enumerate}
\end{theorem}

\begin{figure}			
\begin{center}
\begin{overpic}[width=12cm]{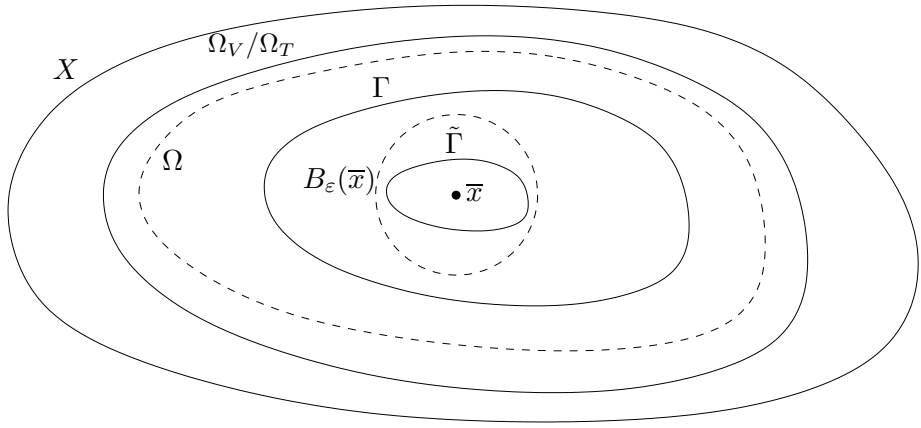}
\put(5,38){\small{$X$}}
\put(22,41){\footnotesize{$\Omega_V/\Omega_T$}}
\put(17,28){\small{$\Omega$}}
\put(40,36){\small{$\Gamma$}}
\put(32.4,26){\small{$B_{\varepsilon}(\overline{x})$}}
\put(48.,29.8){\small{$\tilde{\Gamma}$}}
\put(50.3,24.5){\small{$\overline{x}$}}
\end{overpic}
\end{center}
\caption{Domains and sets used in the statement and proof of Theorem \ref{thm:mainresult}. The dotted lines show the boundary of the set $\mathcal{D}=\Omega\setminus B_{\varepsilon}(\overline{x})$ where the Lyapunov functions are approximated. We also have  $\Omega_V,\Omega_T\subset A(\overline{x})$.}		\label{fig:domains}
\end{figure}

The main point is that the expressions on the right hand side of \eqref{eqn:mainresultV}--\eqref{eqn:mainresultT2} can be made arbitrarily small as the data density increases and for suitably chosen $\lambda$ (see equation \eqref{eqn:lambdachoice}). Therefore the orbital derivative of our Lyapunov function approximations $\hat{V}$ and $\hat{T}$ become arbitrarily close in the infinity norm to those of $V$ and $T$ respectively. Estimate \eqref{eqn:mainresultV} implies that the orbital derivative of $\hat{V}$ will be negative in $\mathcal{D}$ (which does not contain a small neighbourhood of the equilibrium $\overline{x}$), since $\langle\nabla V(x), f^*(x)\rangle_{\mathbb{R}^d} = -p(x)$ where $p$ is a positive definite function (see Theorem \ref{thm:VLyapunovconverse}). The analogous statement is true for $\hat{T}$, since $\langle\nabla T(x), f^*(x)\rangle_{\mathbb{R}^d} = -\overline{c} <0$.
In principle the neighbourhood $B_\varepsilon(\overline{x})$ can shrink as the data density increases (as $h_\mathbf{x}$ and $||\mathbf{w}||_{\mathbb{R}^m}$ tend to zero).

The above estimate contains $\lambda >0$ as a regularisation parameter of our algorithm, and $h_\mathbf{q}$ as the fill distance of a set of sampled points in $\Omega_V$ (resp. $\Omega_T$) of our choosing. Similarly, $h_{\tilde{\mathbf{q}}}$ is the fill distance of a set of sampled points on $\Gamma$, which we are able to choose. The constants in the above estimates depend on $d$, $\sigma$, the choice of function spaces for approximation and  the vector field $f^*$.
%The remaining terms in the above estimate are related to the choice of function space, and will be elucidated in \S\ref{sec:RKHS}.

%
%\begin{remark}
%The statement of Theorem \ref{thm:mainresult} is valid in the domain $\mathcal{D}$, where a neighbourhood of the equilibrium has been removed. This is necessary for $\hat{V}$ as our estimates cannot guarantee convergence of our estimate  in arbitrarily small neighbourhoods of $\overline{x}$. Note that the Lyapunov function $T$ is not even defined at $\overline{x}$ and has values tending to $-\infty$ as $x\rightarrow\overline{x}$.
%
%In practice this is not a drawback even for our estimate $\hat{V}$, since the orbital derivative of the true Lyapunov function $V$ is zero at the equilibrium $\overline{x}$. Therefore estimates of the form \eqref{eqn:mainresultV} are not useful in small neighbourhoods of the equilibrium anyway, since it may still be that the orbital derivative of $\hat{V}$ is positive in this neighbourhood.
%
%In principle the neighbourhood $B_\varepsilon(\overline{x})$ can shrink as the data density increases (as $h_\mathbf{x}$ and $||\mathbf{w}||$ tend to zero).
%\end{remark}

The rest of the paper is organised as follows. In \S\ref{sec:conversethms} we provide the converse theorems for the Lyapunov functions $V$ and $T$. In \S\ref{sec:background} we set out the framework for the function spaces that are used to approximate the Lyapunov functions, as well as previous results on the approximation of Lyapunov functions when the right hand side of \eqref{eqn:dynsys} is known. The algorithms themselves that are used to compute $\hat{V}$ and $\hat{T}$ are detailed in \S\ref{sec:algorithm}. In \S\ref{sec:errorf} we provide an estimate for our approximation of the right hand side of \eqref{eqn:dynsys}, which is then used in the proof of Theorem \ref{thm:mainresult}  in \S\ref{sec:proofofmainthm}.

%%%%%%%%%%%%%%%%%%%%%%%%%%%%%%%%%%%%%%%%%%%%%%%%%%%%%%%%%%%%%%%%%%%%%%%%%%%%

\section{Converse theorems for Lyapunov functions}			\label{sec:conversethms}

The concept of a Lyapunov function dates back to 1893, where Lyapunov introduced these functions for the stability analysis of an equilibrium for a given differential equation, without the explicit knowledge of the solutions \cite{Lya07}. Many converse theorems have been proved that guarantee the existence of a Lyapunov function under certain conditions, see \cite{Gie07:a,Hah67,Kel15} for an overview. Massera \cite{Mas49} provided the first main converse theorem for $C^1$ vector fields where $A(\overline{x})=\mathbb{R}^d$, with further developments by several authors to prove the existence of smooth Lyapunov functions under weak smoothness assumptions on the right hand side (see e.g. \cite{Hah59,LinSonWan96,Wes67}).

The existence of a Lyapunov function for system \eqref{eqn:dynsys} with given values of the orbital derivative has been shown by Bhatia \cite{Bha67,BhaSze70}, as stated in the following theorems (see also \cite{Gie07:a}). We also refer to \cite{GieHaf15} for a proof that the conditions on the function $p$ given here are sufficient to define the Lyapunov function $V$, in contrast to the conditions given in \cite{Gie07:a}. First we make the following definition.

\begin{definition}			\label{def:classK}
A continuous function $\alpha: [0,\infty) \rightarrow [0,\infty)$ is a class $\mathcal{K}$ function if $\alpha(0) = 0$ and $\alpha$ is strictly monotonically increasing.
\end{definition}

\begin{theorem}			\label{thm:VLyapunovconverse}
Consider the autonomous system of differential equations
%\begin{equation*}
$\dot{x}=f^*(x)$,
%\end{equation*}
where $f^*\in C^{\nu_1}(\mathbb{R}^d,\mathbb{R}^d)$, $\nu_1\ge1$, $d\in\mathbb{N}$. We assume the system to have an
exponentially asymptotically stable equilibrium $\overline{x}$. Let $p\in C^{\nu_1}(\mathbb{R}^d,\mathbb{R})$ be a function with the following properties:
\begin{enumerate}
\item $p(x)>0$ for $x\ne\overline{x}$, and $p(\overline{x}) = 0$.
%\item $p(x) = O(||x - \overline{x}||^\gamma)$ with $\gamma >0$ for $x\rightarrow\overline{x}$,
%\item For all $\varepsilon >0$, $p$ has a positive lower bound on $\mathbb{R}^d\backslash B_\varepsilon(\overline{x})$.
\item There is a class $\mathcal{K}$ function $\alpha$ such that $p(x-\overline{x}) \ge \alpha(||x-\overline{x}||_2)$ for all $x\in\mathbb{R}^d$.
\end{enumerate}

Then there exists a Lyapunov function $V\in C^{\nu_1}(A(\overline{x}),\mathbb{R})$ (where $A(\overline{x})$ is the basin of attraction of $\overline{x}$), such that
\begin{equation}			\label{eq:Veq}
\langle \nabla V(x), f^*(x)\rangle_{\mathbb{R}^d} = -p(x)
\end{equation}
holds  for all $x\in A(\overline{x})$. The Lyapunov function $V$ is uniquely defined up to a constant.
\end{theorem}

We may also choose $p(x)$ to be a positive constant in equation \eqref{eq:Veq} to obtain a Lyapunov function $T$, defined on $A(\overline{x})\backslash\{\overline{x}\}$, for which $\lim_{x\rightarrow\overline{x}}T(x) = -\infty$.

\begin{theorem}			\label{thm:TLyapunovconverse}
Consider the autonomous system of differential equations
%\begin{equation*}
$\dot{x}=f^*(x)$,
%\end{equation*}
where $f^*\in C^{\nu_1}(\mathbb{R}^d,\mathbb{R}^d)$, $\nu_1\ge1$, $d\in\mathbb{N}$. We assume the system to have an
exponentially asymptotically stable equilibrium $\overline{x}$. Then for all $\overline{c}\in\mathbb{R}^+$, there exists a Lyapunov function $T\in C^{\nu_1}(A(\overline{x})\,\backslash\,\{\overline{x}\},\mathbb{R})$ such that
\begin{equation}			\label{eq:Teq}
\langle \nabla T(x), f^*(x)\rangle_{\mathbb{R}^d} = -\overline{c}.
\end{equation}
Moreover, $\lim_{x\rightarrow\overline{x}}T(x)=-\infty$.
\end{theorem}

The Lyapunov function $T$ will be uniquely defined if its values  are given on a non-characteristic hypersurface $\Gamma\subset A(\overline{x})$ \cite{Gie07:a} by a function $\xi_T\in C^{\nu_1}(\Gamma,\mathbb{R})$; that is, $T(x) = \xi_T(x)$ for $x\in \Gamma$.

\begin{definition}[Non-characteristic hypersurface]			\label{def:noncharhyp}
Consider $\dot{x}=f^*(x)$, where $f^*\in C^{\nu_1}(\mathbb{R}^d,\mathbb{R}^d)$, $\nu_1\ge1$, $d\in\mathbb{N}$. Let $h\in C^{\nu_1}(\mathbb{R}^d,\mathbb{R})$ and recall $\varphi_{f^*}: \mathbb{R}^d \times \mathbb{R} \rightarrow \mathbb{R}^d$ denotes the flow mapping. The set $\Gamma\subset\mathbb{R}^d$ is
called a non-characteristic hypersurface if
\begin{enumerate}
\item $\Gamma$ is compact,
\item $h(x) = 0$ if and only if $x\in\Gamma$,
\item $h'(x) = \langle h(x),f^*(x)\rangle_{\mathbb{R}^d}<0$ holds for all $x\in\Gamma$, and
\item for each $x\in A(\overline{x})\backslash\{\overline{x}\}$ there is a time $\theta(x)\in\mathbb{R}$ such that $\varphi_{f^*}(x,\theta(x))\in\Gamma$.
\end{enumerate}
\end{definition}

An example of a non-characteristic hypersurface is the level set of a (local) Lyapunov function, see \cite{Gie07:a}. In what follows we assume that we have chosen a  non-characteristic hypersurface $\Gamma$ together with a function $\xi_T\in C^{\nu_1}(\Gamma,\mathbb{R})$ so that $T$ is uniquely defined.

\begin{remark}
The Lyapunov function $V$ is smooth and defined on the entire basin of attraction $A(\overline{x})$. However, its orbital derivative vanishes at the equilibrium $\overline{x}$, and therefore estimates that bound the error of the orbital derivative for numerical approximations of $V$ cannot guarantee negative orbital derivative arbitrarily close to the equilibrium $\overline{x}$. On the other hand, the Lyapunov function $T$ is not even defined at $\overline{x}$ and is unbounded near the equilibrium. However, its definition has the advantage that it is not required that we know where the equilibrium is.
%Numerical approximation of $T$ will additionally require the definition of a non-characteristic hypersurface $\Gamma$ and defined values for $T(x), x\in\Gamma$.
\end{remark}

In our approach to approximate Lyapunov functions directly from data, we will provide estimates for the approximation of both $V$ and $T$. The strategy is to first approximate $f^*$ from the data, and use this in turn to approximate the Lyapunov function.

%%%%%%%%%%%%%%%%%%%%%%%%%%%%%%%%%%%%%%%%%%%%%%%%%%%%%%%%%%%%%%%%%%%%%%%%%%%%

\section{Function spaces and approximation theorems for Lyapunov functions}			\label{sec:background}			

\subsection{Reproducing kernel Hilbert spaces}		\label{sec:RKHS}

The function spaces that we  use to search for our approximations to both $f^*$ and the Lyapunov functions $V$ and $T$ will be  reproducing kernel Hilbert spaces (RKHS). For a survey of the main properties of RKHS spaces mentioned in this section, we refer to \cite{CucSma01}.

In order to define an RKHS function space we first fix a continuous, symmetric, positive definite function (a ``kernel'') $K:X\times X\rightarrow\mathbb{R}$, and set $K_x:= K(\cdot,x)$. Define the Hilbert space $\mathcal{H}_K$ by first considering all finite linear combinations of functions $K_x$, that is $\sum_{x_i\in X} a_iK_{x_i}$ with finitely many $a_i\in\mathbb{R}$ nonzero. An inner product $\langle \cdot , \cdot \rangle_K$ on this space is defined by
\begin{equation*}
\langle K_{x_i},K_{x_j} \rangle_K := K(x_i,x_j)
\end{equation*}
and extending linearly. One takes the completion to obtain $\mathcal{H}_K$.

Alternatively, an equivalent definition of an RKHS is as a Hilbert space of real-valued functions on $X$ for which the evaluation functional $\delta_x(f):= f(x)$ is continuous for all $x\in X$.

Finite dimensional subspaces of $\mathcal{H}_K$ can also be naturally defined by taking a finite number of points $\mathbf{x}:=\{x_1,\ldots,x_m\}\subset X$ and considering the linear span
\begin{equation*}
\mathcal{H}_{K,\mathbf{x}} := \textrm{sp}\{K_x:x\in\mathbf{x}\}.
\end{equation*}
In practice we will seek functions in these finite dimensional subspaces as approximations for $f^*$.

Within these Hilbert spaces the reproducing property holds:
\begin{equation}
\langle K_x,f\rangle_K = f(x), \qquad \forall f\in\mathcal{H}_K.		\label{eq:reproducing}
\end{equation}
If we denote $\kappa:=\sqrt{\sup_{x\in X} K(x,x)}$, then $\mathcal{H}_K\subset C(X)$ and it follows that
\begin{equation}
||f||_{L^\infty(X)} \le \kappa||f||_K, \qquad \forall f\in\mathcal{H}_K.		\label{eq:supnorminclusion}
\end{equation}
%In addition, in \cite{SmaZho07,Zho03} it is shown that when $K\in C^{2s+\varepsilon}(X\times X)$ with $0<\varepsilon<2$, the inclusion $\mathcal{H}_K\subset C^{s+\varepsilon/2}(X)$ is well defined and bounded, and we have
%\begin{equation}
%||f||_{C^s(X)} \le 4^s ||K||^{1/2}_{C^{2s}}||f||_K,\qquad\forall f\in\mathcal{H}_K.		\label{eq:sobolevinclusion}
%\end{equation}

The RKHS $\mathcal{H}_K$ can also be defined by means of an integral operator. Let $\rho$ be any (finite) strictly positive Borel measure on $X$ (e.g. Lebesgue measure) and ${L}^2_\rho(X)$ be the Hilbert space of square integrable functions on $X$. Then define the linear operator $L_K:{L}^2_\rho(X)\rightarrow C(X)$ by
\begin{equation}			\label{eq:LK}
(L_K f)(x) = \int_X K(x,y)f(y)d\rho(y).
\end{equation}
When composed with the inclusion $C(X)\hookrightarrow L^2_\rho(X)$ we obtain an operator from $L^2_\rho(X)$ to $L^2_\rho(X)$, which we also denote by $L_K$.
 $L_K$ is then a self-adjoint compact operator, and  it is also positive if the kernel $K$ is positive definite. Also the map
\begin{equation*}
L_K^{1/2}:{L}^2_\rho(X)\rightarrow \mathcal{H}_K
\end{equation*}
defines an isomorphism of Hilbert spaces. $L_K^{1/2}$ is well defined as an operator on ${L}^2_\rho(X)$ in the sense that $L_K = L_K^{1/2}\circ L_K^{1/2}$.

\subsection{Sobolev space RKHS}			\label{sec:SobolevspaceRKHS}

In this paper we will work with reproducing kernel Hilbert spaces that are Sobolev spaces. Given the open domain $\mathcal{B}\subset \mathbb{R}^d$, for $k\in \mathbb{N}_0$, $1\le p < \infty$, the Sobolev space $W^k_p(\mathcal{B})$ consists of all functions $f$ with weak derivatives $D^\alpha f \in L^p(\mathcal{B})$, $|\alpha| \le k$. We also use the following notation to define the (semi-)norms
\begin{equation*}
|f|_{W^k_p(\mathcal{B})} = \left( \sum_{|\alpha | = k} || D^\alpha f ||^p_{L^p(\mathcal{B})} \right)^{1/p},\quad
||f||_{W^k_p(\mathcal{B})} = \left( \sum_{|\alpha | \le k} || D^\alpha f ||^p_{L^p(\mathcal{B})} \right)^{1/p}.
\end{equation*}
With $p=\infty$ the norms are defined in the natural way:
\begin{equation*}
|f|_{W^k_\infty (\mathcal{B})} = \sup_{\alpha = k} ||D^\alpha f ||_{L^\infty(\mathcal{B})} \quad \text{and}\quad ||f||_{W^k_\infty (\mathcal{B})} = \sup_{\alpha \le k} ||D^\alpha f ||_{L^\infty(\mathcal{B})}.
\end{equation*}
We will also use fractional order Sobolev spaces. For a detailed discussion see e.g. \cite{Ada75}.

The Sobolev embedding theorem states that for $\tau > d/2$, $W^\tau_2(\mathbb{R}^d)$ can be embedded into $C(\mathbb{R}^d)$, and therefore it follows that $W^\tau_2(\mathbb{R}^d)$ is a RKHS, using the fact that the pointwise evaluation functional is then continuous. Several kernel functions are known to generate RKHS spaces that are norm-equivalent to Sobolev spaces \cite{Opf06:a,Opf06:b}. We will choose to work with the Wendland functions \cite{Wen95}. These are positive definite, compactly supported radial basis function kernels that are represented by a univariate polynomial on their support.
\begin{definition}[Wendland function]
Let $l\in\mathbb{N}$, $k\in\mathbb{N}_0$. We define by recursion
\begin{eqnarray*}
\psi_{l,0}(r) &=& (1-r)^l_+\\
\textrm{and }\psi_{l,k+1}(r) &=& \int^1_r t\psi_{l,k}(t)dt
\end{eqnarray*}
for $r\in\mathbb{R}^+_0$. Here, $x_+ = x$ for $x\ge 0$ and $x_+ = 0$ for $x<0$.
\end{definition}

Setting $l:=\lfloor\frac{d}{2}\rfloor + k + 1$, the Wendland functions are characterised by a smoothness index $k\in\mathbb{N}$, and belong to $C^{2k}(\mathbb{R}^d)$. For a domain  $D\subset \mathbb{R}^d$ with a Lipschitz boundary, the Wendland  function kernel is given by $K(x,y):=\psi_{l,k}(c||x-y||_{\mathbb{R}^d})$, $c>0$, for $x,y\in D$. The Wendland function kernel generates an RKHS consisting of the same functions as  the Sobolev space $W^\tau_2(D)$ with $\tau = k + (d+1)/2$, with an equivalent norm \cite[Corollary 10.48]{Wen05}. Therefore the generated Sobolev space is of integer order when $d$ is odd, and integer plus one half when $d$ is even.

From now on we shall use RKHS spaces generated by Wendland function kernels. We will use two such RKHS spaces for the two parts of our algorithm: to approximate the vector field $f^*$ in \eqref{eqn:dynsys} we use the space $\mathcal{H}_{K^{1}}$ defined on $X$, corresponding to the Wendland function kernel $K^{1}$ with smoothness index $k_1$, such
 that $\tau_1 = k_1 + (d+1)/2$ with $\lceil \tau_1 \rceil = \nu_1$. Then $\mathcal{H}_{K^1}$ is norm-equivalent to $W^{\tau_1}_2(X)$. In this case, when $d$ is odd we have that $\lceil\tau_1\rceil=\tau_1$, and the assumption that $\nu_1 \ge (3d+7)/2$  implies that $k_1 \ge d+3$. When $d$ is even, $\lceil\tau_1\rceil = \tau_1 + 1/2$, and the assumption that $\nu_1\ge (3d+12)/2$  gives  $k_1 \ge d+5$ (cf. Theorem \ref{thm:mainresult}).

In the second part of our algorithm we approximate the Lyapunov function. For this, we use the RKHS space $\mathcal{H}_{K^{2}}$ defined on $\Omega_V$ (resp. $\Omega_T$) corresponding to the Wendland function kernel $K^{2}$ with smoothness index $k_2$, such that $k_2 = k_1 - (d+2)$ if $d$ is odd, or $k_2 = k_1 - (d+3)
$ if $d$ is even. Correspondingly, this implies that $k_2 \ge 1$ when $d$ is odd, and $k_2 \ge 2$ when $d$ is even. Here, $\mathcal{H}_{K^2}$ is norm-equivalent to $W^{\tau_2}_2(\Omega_V)$ (resp.  $W^{\tau_2}_2(\Omega_T)$), where $\tau_2:= k_2 + (d+1)/2$.

% Then for a domain  $X\subset \mathbb{R}^d$ with a Lipschitz boundary,   $\mathcal{H}_K$ consists of the same functions as $W^\tau_2(X)$ with an equivalent norm (see also \cite[Lemma 3.6]{GieHaf15}).

%\begin{remark}		\label{rmk:differentRKHS}
%The choice of the Sobolev space RKHS $W_2^\tau(\mathbb{R}^d)$ is only necessary in the second algorithm (see \S\ref{sec:algorithm}) in order to prove our estimates \eqref{eqn:mainresultV}--\eqref{eqn:mainresultT2}.
%In principle it is possible to choose a different RKHS for  Algorithm \ref{alg:approxf}, for example the Gaussian RKHS.
%\end{remark}

These function spaces consist of smooth functions as a consequence of the following generalised Sobolev inequality (for a proof, see e.g. \cite[Chapter 5.7, Theorem 6]{Eva98}).

\begin{lemma}					\label{lem:genSobolev}
Let $\mathcal{B}\subset\mathbb{R}^d$ be a bounded open set with $C^1$ boundary. For $u\in W^{m}_2(\mathcal{B})$ where $m > d/2$, we have
\begin{equation}				\label{eq:genSobolev}
|| u ||_{C^{m-\lfloor{\frac{d}{2}}\rfloor - 1,\gamma}(\mathcal{B})} \le C ||u||_{W^{m}_2(\mathcal{B})},
\end{equation}
where $\gamma = \frac{1}{2}$ if $d$ is odd, and $\gamma$ is any element in $(0,1)$ if $d$ is even.
\end{lemma}

\begin{corollary}				\label{cor:smoothnessK1K2}
For $f\in\mathcal{H}_{K^1}$ and $g\in\mathcal{H}_{K^2}$, we have
\begin{eqnarray}
|| f ||_{C^{k_1}(X)} &\le& C || f ||_{K^1},	\quad (\textrm{when }d\textrm{ is odd}),	\label{eqn:smoothnessK1:a}\\
|| f ||_{C^{k_1-1}(X)} &\le& C || f ||_{K^1},	\quad (\textrm{when }d\textrm{ is even}),	\label{eqn:smoothnessK1:b}\\
|| g ||_{C^{1}(\Omega_V/\Omega_T)} &\le& C || g ||_{K^2}.		\label{eqn:smoothnessK2}
\end{eqnarray}
%where $\nu_2:= k_1$ if $d$ is odd, or $\nu_2:= k_1-1$ if $d$ is even.
\end{corollary}
\begin{proof}
Following the arguments given in \cite{NarWarWen04}, there exists a bounded extension operator $E$ that extends $f\in W_2^{\tau_1}(X)$ to $Ef\in W_2^{\tau_1}(\mathbb{R}^d)$ such that $Eu = u$ on $X$. Then, from Lemma \ref{lem:genSobolev} we have
\begin{eqnarray*}
||f||_{C^{\tau_1-\lfloor{\frac{d}{2}}\rfloor - 1,\gamma}(X)} &=& ||Ef||_{C^{\tau_1-\lfloor{\frac{d}{2}}\rfloor - 1,\gamma}(X)}
 \le  || Ef ||_{C^{\tau_1-\lfloor{\frac{d}{2}}\rfloor - 1,\gamma}(\mathbb{R}^d)}\\
& \le & \tilde{C} || Ef ||_{W^{\tau_1}_2(\mathbb{R}^d)}
 \le  C || f ||_{W^{\tau_1}_2(X)}.
\end{eqnarray*}
Then \eqref{eqn:smoothnessK1:a} and \eqref{eqn:smoothnessK1:b} follow from using the norm equivalence of $\mathcal{H}_{K^1}$  to $W^{\tau_1}_2(X)$, and that $\tau_1 = k_1 + (d+1)/2$. Inequality \eqref{eqn:smoothnessK2} follows similarly, also using  $k_2 \ge 1$ when $d$ is odd, and $k_2 \ge 2$ when $d$ is even.
\end{proof}

%Therefore, for a Wendland function kernel with smoothness index $k$, we have that $\mathcal{H}_K$ consists of functions that lie in the H\"{o}lder space $C^{k,1/2}(X)$ if $d$ is odd, and in $C^{k-1,\gamma}(X)$ (for any $\gamma\in(0,1)$) if $d$ is even.

%Then we have $\mathcal{H}_{K^1} \subset C^{\nu_2}(X)$, where
%Furthermore, we have in both cases ($d$ odd or even), that the functions belonging to the space $\mathcal{H}_{K^2}$ will be at least $C^1$. This is precisely the motivation for the smoothness assumptions.

\subsection{Generalised interpolant and approximation theorems}

In this section we introduce the generalised interpolant that is used to approximate the Lyapunov functions $V$ and $T$.
%We  recall some results and definitions from \cite{GieWen07}, which provides estimates on the convergence of Lyapunov function approximations for known vector fields $g\in C^\nu(\mathbb{R}^d,\mathbb{R}^d)$.
Consider a  general interpolation setting where $\Omega\subset\mathbb{R}^d$ is a bounded domain having a Lipschitz boundary. Let  $L$ be a linear differential operator and $\mathbf{q}:=\{q_1,\ldots,q_M\}\subset\Omega$ be a set of pairwise distinct points which do not contain any singular points of $L$. (A point $q\in\mathbb{R}^d$ is a singular point of $L$ if $\delta_q\circ L = 0$, see also \cite{GieWen07}.) We define linear functionals
\begin{equation*}
\lambda^j(u) := \delta_{q_j}\circ L(u) = Lu(q_j).
\end{equation*}
%We wish to find a function $s\in W_2^\tau(\mathbb{R}^d)$ that satisfies $Ls(q_i) = \beta_i$ for given $\beta_i$.

\begin{definition}[Generalized interpolant]			\label{def:genint}
Suppose we have values $Lu(q_i) = \beta_i$ for $i=1,\ldots,M$ for a function $u:\Omega\rightarrow\mathbb{R}$.
Given a sufficiently smooth kernel $K(\cdot,\cdot)$, define the generalised interpolant as
\begin{equation*}
s = \sum_{j=1}^M \alpha_j(\delta_{q_j} \circ L)^yK(\cdot,y),
\end{equation*}
where $(\delta_{q_j} \circ L)^y$ is the linear function applied to one argument of the kernel. The coefficient vector $\alpha$ is the solution of $A\alpha = \beta = (\beta_i)$ with the interpolation matrix $A \in\mathbb{R}^{M\times M}$  given by
\begin{equation*}
(A)_{ij} = (\delta_{q_i}\circ L)^x(\delta_{q_j}\circ L)^yK(x,y).
\end{equation*}
\end{definition}
\begin{remark}		\label{rem:sgenint}
According to the above definition, we have $Ls(q_i)=\beta_i$. In addition, it can be shown that the generalised interpolant above is the unique norm-minimal  interpolant in $\mathcal{H}_K$, see \cite{Gie07:a,GieWen07}. The matrix $A$ is guaranteed to be invertible due to our  choice of the Wendland function kernel $K$ \cite[Section 3.2.2]{Gie07:a}, and since $\mathbf{q}$ does not contain any singular points.
\end{remark}

We conclude this section by citing  the following theorem from \cite{GieWen07}, which provides convergence estimates for approximating  Lyapunov functions $V_g$ (resp. $T_g$) with the generalised interpolants $s_1$ (resp. $s_2$) as above, for a known dynamical system $\dot{x} = g(x)$.
\begin{theorem}				\label{thm:GieWen}
Let $\nu_2:= \lceil \tau_2 \rceil$ with $ \tau_2 = k_2 +  (d+1)/2$, where $k_2$ is the smoothness index of the compactly supported Wendland function. Let $k_2>1/2$ if $d$ is odd or $k_2>1$ if $d$ is even. Consider the dynamical system defined by the ordinary differential equation $\dot{x} = g(x)$, where $g\in C^{\nu_2}(\mathbb{R}^d,\mathbb{R}^d)$. Let $\overline{x}\in\mathbb{R}^d$ be an equilibrium such that the real parts of all eigenvalues of $Dg(\overline{x})$ are negative, and suppose $g$ is bounded in $A(\overline{x})$.

Let $\Gamma$ be a non-characteristic hypersurface as in Definition \ref{def:noncharhyp}, with $h\in C^{\nu_2}(\mathbb{R}^d,\mathbb{R})$ and  $\xi_T\in C^{\nu_2}(\Gamma,\mathbb{R})$. Let $V_g\in W^{\tau_2}_2(A(\overline{x}),\mathbb{R})$, $T_g\in W_2^{\tau_2}(A(\overline{x}) \backslash \{\overline{x}\},\mathbb{R})$ be the Lyapunov functions of Theorems \ref{thm:VLyapunovconverse} and \ref{thm:TLyapunovconverse} for the system  $\dot{x} = g(x)$, with $T_g(x)=\xi_T(x)$ for $x\in\Gamma$.

 Given pairwise distinct sets of points $\mathbf{q}:=(q_i)_{i=1}^M$ in $A(\overline{x})$ (not containing $\overline{x}$) and $\tilde{\mathbf{q}}:=(q_i)_{i=M+1}^{M+N}$ in $\Gamma$,
 and let $\Omega\subset A(\overline{x})$ be a bounded domain with Lipschitz boundary, with $\mathbf{q}\subset \Omega$.
% define $\Omega_{V_g} := \{x\in A(\overline{x}) \mid V_g(x) \le R  \textrm{ and }h(x)\ge0\}$ (resp. $\Omega_{T_g} := \{x\in A(\overline{x})\backslash\{\overline{x}\} \mid T_g(x) \le R \textrm{ and }h(x)\ge0\}$) with $R>0$ large enough so that $\mathbf{q}\subset\Omega_{V_g}$ (resp. $\Omega_{T_g}$).
 Let $s_1$ and $s_2$ be the generalised interpolants satisfying
\begin{eqnarray*}
\lambda_g^i(s_1)& := &\langle\nabla s_1(q_i), g(q_i)\rangle_{\mathbb{R}^d} = -p(q_i),\qquad i=1,\ldots, M,\\
\lambda_g^i(s_2)& := &\langle\nabla s_1(q_i), g(q_i)\rangle_{\mathbb{R}^d} = -\overline{c}, \qquad i=1,\ldots, M,\\
\lambda^{i,0}(s_2)& := & s_2(q_i) = \xi_T(q_i), \qquad i=M+1,\ldots, M+N,
\end{eqnarray*}
and let the fill distance of the set of points $\mathbf{q}$ in
$\Omega$
%$\Omega_V$ (resp. $\Omega_T$)
and $\tilde{\mathbf{q}}$ in $\Gamma$ be $h_{\mathbf{q}}$ and $h_{\mathbf{\tilde{q}}}$ respectively. Then for $h_{\mathbf{q}}, h_{\mathbf{\tilde{q}}}$ sufficiently small, the following estimates hold:
\begin{eqnarray}
|| \langle\nabla s_1, g\rangle_{\mathbb{R}^d} - \langle\nabla V_g,g\rangle_{\mathbb{R}^d}  ||_{L^\infty(\Omega)} & \le & Ch_{\mathbf{q}}^{k-\frac{1}{2}}||V_g||_{W_2^{\tau_2}(\Omega)},\\
||\langle\nabla s_2, g\rangle_{\mathbb{R}^d} - \langle\nabla T_g,g\rangle_{\mathbb{R}^d} ||_{L^\infty(\Omega)} & \le & Ch_{\mathbf{q}}^{k-\frac{1}{2}}||T_g||_{W_2^{\tau_2}(\Omega)},\\
||s_2 - T_g||_{L^\infty(\Gamma)} & \le & Ch_{\tilde{\mathbf{q}}}^{k+\frac{1}{2}}||T_g||_{W_2^{\tau_2}(\Omega)}.	\label{eqn:ests2gamma}
\end{eqnarray}
\end{theorem}
%\begin{remark}
%The domain $\Omega_{V_g}$ is slightly different to that used in \cite{GieWen07}, but the  statement still holds as $\Omega_{V_g}$ above is a domain with Lipschitz boundary that contains $\mathbf{q}$. We will require the above definition of $\Omega_{V_g}$ in the proof of our main result in Section \ref{sec:proofofmainthm}.
%\end{remark}
%%%%%%%%%%%%%%%%%%%%%%%%%%%%%%%%%%%%%%%%%%%%%%%%%%%%%%%%%%%%%%%%%%%%%%%%%%%%%%%%%%%%%%%%%%

\section{The Algorithm}			\label{sec:algorithm}

Here we present the algorithm for which the estimate given in Theorem \ref{thm:mainresult} holds. The algorithm is actually split into two parts. The first part computes $f_{\mathbf{z},\lambda}$ as an approximation to $f^*$ (Algorithm \ref{alg:approxf}), and the second part computes  $\hat{V}$ or $\hat{T}$ as an approximation to the Lyapunov functions $V$ or $T$ respectively given in Theorems \ref{thm:VLyapunovconverse} and \ref{thm:TLyapunovconverse} (Algorithm \ref{alg:approxV} or \ref{alg:approxT}). As discussed in \S\ref{sec:SobolevspaceRKHS}, we will use two RKHS spaces $\mathcal{H}_{K^{1}}$ and $\mathcal{H}_{K^{2}}$ for the two parts of the algorithm, corresponding to the Wendland function kernels $K^{1}$ and $K^{2}$ with smoothness indices $k_1$ and $k_2$ respectively. We recall that the smoothness indices are chosen such that $\tau_1 = k_1 + (d+1)/2$ with $\lceil\tau_1\rceil = \nu_1$, and $k_2 = k_1 - (d+3)/2$ if $d$ is odd, or $k_2 = k_1 - (d+4)/2$ if $d$ is even.

%The RKHS spaces $\mathcal{H}_{K^{1}}$ and $\mathcal{H}_{K^{2}}$ used for each part are norm-equivalent to the Sobolev spaces $W^{k_1 + (d+1)/2}_2(X)$ and $W^{k_2 + (d+1)/2}_2(\Omega_V)$ (or $W^{k_2 + (d+1)/2}_2(\Omega_T)$) respectively.

Recall that our sampled data values $\mathbf{z}:=(x_i, y_i)_{i=1}^m \in (X\times \mathbb{R}^d)^m$ take the form  $y_i = f^*(x_i) + \eta_{x_i}$, with $\eta_x \in \mathbb{R}^d$ a random variable drawn from a probability distribution with zero mean  and variance $\sigma_x^2$.

Our approximation scheme for $f^*$ employs a regularised least squares algorithm (see e.g. \cite{EvgPonPog00} and its references) to approximate each component $f^{*,k}$, $k=1,\ldots,d$. We also introduce a weighting $\mathbf{w}=\{w_{x_i}\}_{i=1}^m$ corresponding to the Voronoi tessellation associated with the points $\{x_i\}_{i=1}^m$ \cite{Vor08}.

\begin{definition}[Voronoi tessellation]			\label{def:Voronoi}
Let $X\subset\mathbb{R}^d$ be compact. For a set of pairwise distinct points $\mathbf{x} := \{x_i\}_{i=1}^m \in X^m$, the Voronoi tessellation is the collection of pairwise disjoint open sets $\mathcal{V}_i(\mathbf{x})$, $i=1,\ldots,m$ defined by
\begin{equation*}
\mathcal{V}_i(\mathbf{x}) = \{y\in X \, | \, ||x_i - y||_{\mathbb{R}^d} < ||x_j - y||_{\mathbb{R}^d},\, \text{if } i\ne j\}.
\end{equation*}
\end{definition}
The weighting $\mathbf{w}=\{w_{x_i}\}_{i=1}^m$ is then defined by $w_{x_i} = \rho(\mathcal{V}_i(\mathbf{x}))$, where $\rho$ is the strictly positive Borel measure from \S\ref{sec:RKHS}.

\begin{algorithm}		\label{alg:approxf}
Fix a regularisation parameter $\lambda >0 $, and define $D_w \in\mathbb{R}^{m\times m}$ as the diagonal matrix with diagonal elements $w_{x_i}$, $i=1,\ldots,m$. The approximation $f_{\mathbf{z},\lambda}$ for $f^*$ is constructed component-wise. That is, for each $k = 1,\ldots,d$, we approximate the $k$-th component $f^{*,k}$ by $f^k_{\mathbf{z},\lambda}\in\mathcal{H}_{K^1}$, defined by
%the following:
%\begin{equation}				\label{eq:fzlambdak}		
%f^k_{\mathbf{z},\lambda}:=\arg\min_{f\in\mathcal{H}_K}\left\{ \sum_{i=1}^m w_{x_i}(f(x_i)-y^k_i)^2 + \lambda||f||^2_K\right\}.
%\end{equation}
%In the above,
% $y_i^k$ is simply the $k$-th component of $y_i\in\mathbb{R}^d$.
%
% Lemma \ref{thm:fzlambda} (see also Remark \ref{rem:fzlambda}) shows that the solution to \eqref{eq:fzlambdak} is given by
 $f_{\mathbf{z},\lambda}^k=\sum_{i=1}^m a_i K^1_{x_i}$, where the coefficients $a:=\{a_i\}_{i=1}^m$ may be calculated as the solution to the matrix equation
\begin{equation*}
 (A_{\mathbf{x}}D_w A_{\mathbf{x}} + \lambda A_{\mathbf{x}})a = A_{\mathbf{x}} D_w y^k.
\end{equation*}
Here  $y^k = (y_i^k)_{i=1}^m$ where $y_i^k$ is simply the $k$-th component of $y_i\in\mathbb{R}^d$, and $A_{\mathbf{x}} \in \mathbb{R}^{m\times m}$ is a symmetric matrix defined by $(A_{\mathbf{x}} )_{i,j} = K^1(x_i,x_j)$.
 \end{algorithm}

We note here that due to our choice of RKHS the matrix $A_\mathbf{x}$ is positive definite \cite[Proposition 3.3]{GieWen07}, and therefore the matrix $(A_{\mathbf{x}}D_w A_{\mathbf{x}} + \lambda A_{\mathbf{x}})$ is invertible, as it is the sum of two positive definite matrices. The error in our approximation of $f^*$ by $f_{\mathbf{z},\lambda}$ is studied in \S\ref{sec:errorf}. We will show that this error may be bounded in the supremum norm on the domain $X$, depending on the density of the data in $X$ and the choice of regularisation parameter $\lambda$.

Once we have our approximation $f_{\mathbf{z},\lambda}$, then we construct our Lyapunov function approximation with  the generalised interpolant as in  Theorem \ref{thm:GieWen}, where we set $g = f_{\mathbf{z},\lambda}$. %Then the linear functionals from Theorem \ref{thm:GieWen} becomes $\lambda_{f_{\mathbf{z},\lambda}}^i := \langle \nabla v(q_i), f_{\mathbf{z},\lambda}(q_i)\rangle_{\mathbb{R}^d}$.

Therefore we have the following Algorithm  \ref{alg:approxV} for $\hat{V}$, or Algorithm \ref{alg:approxT} for $\hat{T}$. These algorithms involve sampling our approximation $f_{\mathbf{z},\lambda}$ at a discrete set of points.

\begin{subtheorem}{algorithm}
\begin{algorithm}[Approximation of $V$]		\label{alg:approxV}
First run Algorithm \ref{alg:approxf} on the sampled data set $(x_i,y_i)_{i=1}^m$ to compute $f_{\mathbf{z},\lambda}\in(\mathcal{H}_{K^1})^d$.

Define a set of pairwise distinct points $\mathbf{q}:= (q_i)_{i=1}^M$, with $f_{\mathbf{z},\lambda}(q_i)\ne 0$ for all $i=1,\ldots,M$. The approximation $\hat{V}\in\mathcal{H}_{K^2}$ for the Lyapunov function $V$  (see Theorem \ref{thm:VLyapunovconverse}) is
%defined as follows:
%\begin{equation}			\label{eq:Vhat}
%\hat{V} := \arg\min_{v\in\mathcal{H}_K}\left\{ ||v||_K \mid \lambda_i(v) = -p(q_i) \, \text{ for all }\, q_i\in\mathbf{q}\right\},
%\end{equation}
%where $\lambda_i(v) := \langle \nabla v(q_i), f_{\mathbf{z},\lambda}(q_i)\rangle_{\mathbb{R}^d}$ is in the dual space $\mathcal{H}_K^*$ to $\mathcal{H}_K$, and $p(\cdot)$ is as given in Theorem \ref{thm:VLyapunovconverse}.
%
%From \cite[Theorem 16.1]{Wen05}, the solution to \eqref{eq:Vhat} is
given by $\hat{V}= \sum_{i=1}^M b_i \lambda_{f_{\mathbf{z},\lambda}}^{i,y} K^2(\cdot, y)$, where $\lambda_{f_{\mathbf{z},\lambda}}^{i,y}$ is the linear functional $\lambda_{f_{\mathbf{z},\lambda}}^i$ applied to one argument of the kernel. The coefficients $b:=\{b_i\}_{i=1}^M$ are given by
\begin{equation*}
B_\mathbf{q} b = - p.
\end{equation*}
Here $B_\mathbf{q} \in \mathbb{R}^{M\times M}$ is a symmetric matrix defined by $(B_\mathbf{q})_{i,j} =\lambda_{f_{\mathbf{z},\lambda}}^{i,x}\lambda_{f_{\mathbf{z},\lambda}}^{j,y} K^2(x,y)$ and $p = (p(q_i))_{i=1}^M$.
\end{algorithm}

As before, the choice of RKHS guarantees that the matrix $B_\mathbf{q}$ will be positive definite, provided that $f_{\mathbf{z},\lambda}(q_i) \ne 0$ for all $i=1,\ldots, M$.

To approximate $T$, we  assume that a non-characteristic hypersurface for $f^*$ has been defined as $\Gamma = \{x\in A(\overline{x})\backslash\{\overline{x}\} \mid h(x) = 0\}$ according to Definition \ref{def:noncharhyp}, for which $T(x) = \xi_T(x)$ on $\Gamma$, $\xi_T\in C^{\nu_1}(\Gamma,\mathbb{R})$.

\begin{algorithm}[Approximation of $T$]		\label{alg:approxT}
First run Algorithm \ref{alg:approxf} on the sampled data set $(x_i,y_i)_{i=1}^m$ to compute $f_{\mathbf{z},\lambda}\in(\mathcal{H}_{K^1})^d$.

Define a set of pairwise distinct points $\mathbf{q}:= (q_i)_{i=1}^{M+N}$, with $f_{\mathbf{z},\lambda}(q_i)\ne 0$ for all $i=1,\ldots,M$, and $q_i\in\Gamma$ for $i=M+1,\ldots,M+N$.
% Let $\tilde{\Omega}:=\Omega\cap\{x\in A(\overline{x})\mid h(x)\ge 0\}$, and $r>0$ is large enough so that $\{x \in A(\overline{x})\backslash\{\overline{x}\} \mid T(x)=r\}\cap\Gamma = \emptyset$. In addition to $\mathbf{q}\subset \Omega$, let $\tilde{\mathbf{q}}:=\{q_{M+1},\ldots,q_{M+N}\}$ be a discrete set of pairwise distinct points in $\Gamma$.
The approximation $\hat{T}\in\mathcal{H}_{K^2}$ for the Lyapunov function $T$ (see Theorem \ref{thm:TLyapunovconverse}) is given by $\hat{T} = \sum_{i=1}^M c_i\lambda_{f_{\mathbf{z},\lambda}}^{i,y}K^2(\cdot,y) + \sum_{i=M+1}^{M+N} c_i K^2(\cdot,q_i)$, where the coefficients $c:=\{c_i\}_{i=1}^{M+N}$ are computed as the solution to the matrix equation
\begin{equation*}
C_{\mathbf{q},\tilde{\mathbf{q}}}c = \beta,\quad\text{where }C_{\mathbf{q},\tilde{\mathbf{q}}}:=
\left(\begin{array}{cc}
C & D\\
D^T & C^0
\end{array}\right) \in \mathbb{R}^{(M+N)\times(M+N)},
\end{equation*}
and the submatrices $C \in\mathbb{R}^{M\times M}$, $D\in \mathbb{R}^{M\times N}$ and $C^0\in\mathbb{R}^{N\times N}$ have elements defined by $(C)_{i,j} = \lambda_{f_{\mathbf{z},\lambda}}^{i,x}\lambda_{f_{\mathbf{z},\lambda}}^{j,y}K^2(x,y)$, $(D)_{i,j-N} = \lambda_{f_{\mathbf{z},\lambda}}^{i,y}K^2(q_j,y)$ and $(C^0)_{i-N,j-N} = K^2(q_i,q_j)$. The vector $\beta$ is given by $\beta_i = -\overline{c}$, $i=1,\ldots,M$ and $\beta_i = \xi_T(q_i)$, $i=M+1,\ldots,M+N$.
\end{algorithm}
\end{subtheorem}

 It may again be shown that the matrix $C_{\mathbf{q},\tilde{\mathbf{q}}}$ will be positive definite, providing that $f_{\mathbf{z},\lambda}(q_i) \ne 0$ for all $i=1,\ldots,M$, for details see \cite[Section 3.2.2]{Gie07:a}.

 Our error in the approximations $\hat{V}$ and $\hat{T}$ will depend primarily on the error induced by Algorithm \ref{alg:approxf}, which in turn depends on the density of the data, as well as the regularisation parameter. In addition, there will be an error due to the discrete sampling of $f_{\mathbf{z},\lambda}$ in Algorithms \ref{alg:approxV} and \ref{alg:approxT}. This error will depend on the density of the sample points $\mathbf{q}$ (chosen by the user), which in principle can be entirely independent of the original set of points $\mathbf{x}$ provided by the data. The overall error is the subject of \S\ref{sec:proofofmainthm}, which will prove the estimate given in Theorem \ref{thm:mainresult}.

 %%%%%%%%%%%%%%%%%%%%%%%%%%%%%%%%%%%%%%%%%%%%%%%%%%%%%%%%%%%%%%%%%%%%%%%%%%
 %%%%%%%%%%%%%%%%%%%%%%%%%%%%%%%%%%%%%%%%%%%%%%%%%%%%%%%%%%%%%%%%%%%%%%%%%%
  %%%%%%%%%%%%%%%%%%%%%%%%%%%%%%%%%%%%%%%%%%%%%%%%%%%%%%%%%%%%%%%%%%%%%%%%%%

 \section{Error estimate for $ ||f^k_{\mathbf{z},\lambda} - f^{*,k}||_{L^\infty(X)}$}		\label{sec:errorf}

In this section we estimate the error $||f^k_{\mathbf{z},\lambda} - f^{*,k}||_{L^\infty(X)}$ for each $k=1,\ldots,d$.

We have sampled data of the form $(x_i, y_i)$ in $X\times \mathbb{R}^d$, $i=1,\dots,m$, with $y_i = f^*(x_i) + \eta_{x_i}$. We assume that the one-dimensional random variables $\eta_{x_i}^k \in \mathbb{R}^d$, where $i=1,\dots,m$ and $k=1,\dots,d$, are independent random variables drawn from a probability distribution with zero mean and variance $(\sigma_{x_i}^k)^2$ bounded by $\sigma^2$.

In order to ease notation, and since each $f^k_{\mathbf{z},\lambda}$ is calculated independently for each $k$,  we shall henceforth drop the superscript $k$ and consider the data to be of the form $\mathbf{z}:=(x_i, y_i)_{i=1}^m \in (X\times \mathbb{R})^m$.
Note that in this section we shall only be working with the RKHS $\mathcal{H}_{K^1}$.

%In addition, throughout this section we shall only be working with the RKHS $\mathcal{H}_{K^1}$. We will further simplify notation in this section by dropping the index $1$, and write simply $K$, $\mathcal{H}_K$, $L_K$ and $||\cdot ||_K$.

The following operator definition will enable convenient function representations (c.f. \cite{SmaZho04,SmaZho05}).

\begin{definition}[Sampling Operator]
Given a set $\mathbf{x}:=(x_i)_{i=1}^m$ of pairwise distinct points in $\mathbb{R}^d$, the sampling operator $S_{\mathbf{x}}:\mathcal{H}_{K^1}\rightarrow \mathbb{R}^m$ is defined as
\begin{equation*}
S_{\mathbf{x}}(f) = (f(x_i))_{i=1}^m.
\end{equation*}
\end{definition}
The adjoint operator $S_{\mathbf{x}}^*$ can also be derived as follows. Let $c\in \mathbb{R}^m$, then we have
\begin{equation*}
\left\langle f,S^*_{\mathbf{x}}c\right\rangle_{K^1} = \left\langle S_{\mathbf{x}} f,c\right\rangle_{\mathbb{R}^m} = \sum_{i=1}^m c_if(x_i) = \left\langle f, \sum_{i=1}^m c_i K^1_{x_i}\right\rangle_{K^1},\quad \forall f\in\mathcal{H}_{K^1}.
\end{equation*}
The final equality follows from the reproducing property \eqref{eq:reproducing}. So then $S^*_{\mathbf{x}} c = \sum_{i=1}^m c_iK^1_{x_i}$ for all $c\in \mathbb{R}^m$.

The following Lemma shows that the function $f_{\mathbf{z},\lambda}$ calculated in Algorithm \ref{alg:approxf} is the minimiser of a regularised cost function. We omit the proof, which is similar to that contained in \cite[Theorem 1]{SmaZho05}, except for the introduction of the weights $w$. %We include the proof here for completeness.

\begin{lemma}		\label{lem:fzlambda}
Let
\begin{equation}			\label{eq:fzlambda}
f_{\mathbf{z},\lambda}:=\underset{f\in\mathcal{H}_{K^1}}{\arg\min}\left\{ \sum_{i=1}^m w_{x_i}(f(x_i)-y_i)^2 + \lambda||f||^2_{{K^1}}\right\}.
\end{equation}
%The representer theorem (see e.g. \cite{SchHerSmo01}) says that the function $f_{\mathbf{z},\lambda}$ will take the form $f_{\mathbf{z},\lambda}=\sum_{i=1}^m a_i {K^1}_{x_i}$. We may obtain a convenient expression for $f_{\mathbf{z},\lambda}$ with the following operator definition :
Let $D_w \in \mathbb{R}^{m\times m}$ be the diagonal matrix with entries $\{w_{x_i}\}_{i=1}^m$, and let $\mathbf{y}:=\{y_i\}_{i=1}^m$. If $S_{\mathbf{x}}^* D_w S_{\mathbf{x}} + \lambda I$ is invertible, then $f_{\mathbf{z},\lambda}$ exists and is given by
\begin{equation}
f_{\mathbf{z},\lambda} = J  \mathbf{y},\quad J:=(S_{\mathbf{x}}^* D_w S_{\mathbf{x}} + \lambda I)^{-1}S_{\mathbf{x}}^* D_w.	\label{eqn:Ldef}
\end{equation}
Furthermore, if $f_{\mathbf{z},\lambda}=\sum_{i=1}^m a_i {K^1}_{x_i}$, then the coefficients $a:=\{a_i\}_{i=1}^m$ may be calculated as
\begin{equation}
a = (A_{\mathbf{x}}D_w A_{\mathbf{x}} + \lambda A_{\mathbf{x}})^{-1} A_{\mathbf{x}} D_w \mathbf{y},			\label{eqn:coeffsa}
\end{equation}
where $A_{\mathbf{x}} \in \mathbb{R}^{m\times m}$ is the symmetric matrix defined by $(A_{\mathbf{x}} )_{i,j} = {K^1}(x_i,x_j)$.
\end{lemma}
%\begin{proof}
%Note that $\sum_{i=1}^m w_{x_i}(f(x_i))^2 = ||D_w^{1/2}S_{\mathbf{x}} f||^2_{\mathbb{R}^m} = \langle S_{\mathbf{x}}^*D_w S_{\mathbf{x}} f,f\rangle_{K^1}$. Then we have
%\begin{equation*}
%\sum_{i=1}^m w_{x_i}(f(x_i)-y_i)^2 + \lambda||f||^2_{K^1} = \langle (S_{\mathbf{x}}^*D_w S_{\mathbf{x}} + \lambda I)f,f\rangle_{K^1} - 2\langle S_{\mathbf{x}}^* D_w y , f\rangle_{K^1} + ||D_w^{1/2} y||^2_{\mathbb{R}^m}.
%\end{equation*}
%In order for $f_{\mathbf{z},\lambda}$ to minimise the above quantity, the functional derivative for $f\in\mathcal{H}_{K^1}$ must be zero \cite{PogSma03}. Therefore we have that $f_{\mathbf{z},\lambda}$ satisfies
%\begin{equation*}
%(S_{\mathbf{x}}^* D_w S_{\mathbf{x}} + \lambda I)f_{\mathbf{z},\lambda} = S_{\mathbf{x}}^* D_w y.
%\end{equation*}
%Therefore if $S_{\mathbf{x}}^* D_w S_{\mathbf{x}} + \lambda I$ is invertible, $f_{\mathbf{z},\lambda} = J  y$ and this proves the lemma.
%\qquad\qquad\end{proof}
%
%\begin{remark}		\label{rem:fzlambda}
%From the above arguments it follows that if $f_{\mathbf{z},\lambda}=\sum_{i=1}^m a_i {K^1}_{x_i}$, then the coefficients $a:=\{a_i\}_{i=1}^m$ may be calculated as
%\begin{equation*}
%a = (A_{\mathbf{x}}D_w A_{\mathbf{x}} + \lambda A_{\mathbf{x}})^{-1} A_{\mathbf{x}} D_w y,
%\end{equation*}
%where $A_{\mathbf{x}} \in \mathbb{R}^{m\times m}$ is a symmetric matrix defined by $(A_{\mathbf{x}} )_{i,j} = {K^1}(x_i,x_j)$.
%\end{remark}

Our strategy to prove convergence of the estimate $f_{\mathbf{z},\lambda}$ to $f^*$ combines and adapts results contained in \cite{SmaZho04,SmaZho05,SmaZho07}. The main difference in our case to the standard assumptions in learning theory is that the data $\mathbf{z} = (x_i,y_i)_{i=1}^m$ is not necessarily generated from an underlying probability distribution. Instead, our data is generated by the underlying dynamical system \eqref{eqn:dynsys}, and the data sites may be situated arbitrarily. They could potentially be chosen deterministically, or indeed may be generated randomly.   The assumptions made in \cite{SmaZho05} correspond to this setting, but we would like to provide estimates in terms of the density of the data. This is why we have introduced the weights $(w_{x_i})_{i=1}^m$ corresponding to the $\rho$-volume Voronoi tessellation (c.f. also \cite{SmaZho04}, where a weighting scheme is introduced in a  different setting).

\begin{definition}[Fill distance]			\label{def:filldistance}
Let $\mathcal{C}\subset\mathbb{R}^d$ be a compact set and  $\mathbf{x}:=\{x_1,\ldots,x_m\}$ be a grid, where $\mathbf{x}\subset \mathcal{C}$. We denote the \emph{fill distance} of $\mathbf{x}$ in $\mathcal{C}$ as
\begin{equation*}
h_{\mathbf{x}} = \max_{y\in \mathcal{C}} \min_{x_i \in \mathbf{x}} ||x_i - y||.
\end{equation*}
In particular, for all $y\in \mathcal{C}$ there is a grid point $x_i \in \mathbf{x}$ such that $||y - x_i|| \le h_{\mathbf{x}}$.
\end{definition}

In order to provide an estimate for   $||f_{\mathbf{z},\lambda} - f^*||_{K^1}$, we first define $f_{\mathbf{x},\lambda}, f_\lambda\in\mathcal{H}_{K^1}$ by
\begin{eqnarray}
f_{\mathbf{x},\lambda}& := & J(S_{\mathbf{x}}f^*)		\label{eq:fxlambda}\\
\text{and}\qquad f_\lambda & := & \underset{f\in\mathcal{H_{K^1}}}{\arg\min} \left\{ ||f - f^*||^2_\rho + \lambda ||f||^2_{K^1}\right\}	\label{eq:flambda}
\end{eqnarray}
where $||f||^2_\rho =  \int_X |f(x)|^2d\rho(x)$. Recall from Section \ref{sec:RKHS} that $\rho$ is a finite strictly positive Borel measure on $X$.

The function $f_{\mathbf{x},\lambda}$  may be seen as a `noise-free' version of $f_{\mathbf{z},\lambda}$, such that in the case of no noise, i.e. $\eta_x = 0$ for all $x\in X$, then we would have $f_{\mathbf{z},\lambda} = f_{\mathbf{x},\lambda}$. Correspondingly, the function $f_\lambda$ can be seen as a `data-free' limit of $f_{\mathbf{x},\lambda}$, or the limiting function as the data sites $\mathbf{x}$ become arbitrarily dense in $X$. Finally, if $f^*\in\mathcal{H}_{K^1}$, then $f^*$ is the `regularisation-free' version of $f_\lambda$ -- the limit of $f_\lambda$ as $\lambda\rightarrow 0$.

The strategy is to break down the estimate of $||f_{\mathbf{z},\lambda} - f^*||_{L^\infty(X)}$ according to
\begin{eqnarray*}
||f_{\mathbf{z},\lambda} - f^*||_{L^\infty(X)} & = & ||f_{\mathbf{z},\lambda} - f_{\mathbf{x},\lambda} + f_{\mathbf{x},\lambda} - f_\lambda + f_\lambda - f^*||_{L^\infty(X)}\\
& \le &  ||f_{\mathbf{z},\lambda} - f_{\mathbf{x},\lambda}||_{L^\infty(X)} + ||f_{\mathbf{x},\lambda} - f_\lambda||_{L^\infty(X)} + ||f_\lambda - f^*||_{L^\infty(X)}
\end{eqnarray*}
and estimate each of the three terms in the inequality. These three terms correspond to errors incurred by the noise (sample error), the finite set of data sites (integration error) and the regularisation parameter $\lambda$ (regularisation error).

\subsection{Sample error}

An estimate for $||f_{\mathbf{z},\lambda} - f_{\mathbf{x},\lambda}||^2_{K^1}$ for an unweighted approximation scheme is given in Theorem 2 of \cite{SmaZho05}. A similar result is given in the following lemma, that incorporates the weights of our  scheme. The proof follows that of \cite{SmaZho05}, and here we will just sketch the main adaptations. Importantly, our estimate provides convergence of $f_{\mathbf{z},\lambda}$ to $f_{\mathbf{x},\lambda}$ as the data sites $\mathbf{x}$ become more dense, or as the quantity $||\mathbf{w}||_{\mathbb{R}^m}$ tends to zero.

\begin{lemma}			\label{lem:sampleerror}
For $\lambda > 0$,  for every $0 < \delta < 1$, with probability $1 - \delta$, we have
\begin{equation}
||f_{\mathbf{z},\lambda} - f_{\mathbf{x},\lambda}||_{L^\infty(X)} \le \kappa ||f_{\mathbf{z},\lambda} - f_{\mathbf{x},\lambda}||_{K^1} \le \frac{||\mathbf{w}||_{\mathbb{R}^m} \sigma \kappa^2}{\lambda\sqrt{\delta}}	\label{eq:sampleerror}
\end{equation}
where $\sigma^2:={\sup_{x\in X}\sigma^2_{x}}$ and $\kappa^2:={\sup_{x\in X} {K^1}(x,x)}$.
\end{lemma}
\begin{proof} First note that if $\lambda >0$ then $(S^*_{\mathbf{x}} D_w S_{\mathbf{x}} + \lambda I)$ is invertible. This follows since it is the sum of a positive and a strictly positive operator on $\mathcal{H}_{K^1}$. Similar to \cite[Theorem 2]{SmaZho05}, we have
\begin{eqnarray*}
%f_{\mathbf{z},\lambda} - f_{\mathbf{x},\lambda} &=& J(y - S_{\mathbf{x}}f^*)\\
%& = & (S^*_{\mathbf{x}} D_w S_{\mathbf{x}} + \lambda I)^{-1} S^*_{\mathbf{x}} D_w (y - S_{\mathbf{x}}f^*),\\
%\text{but}\quad S^*_{\mathbf{x}} D_w (y - S_{\mathbf{x}}f^*) & = & \sum_{i=1}^m w_{x_i}(y_i - f^*(x_i)) {K^1}_{x_i}\\
%\text{and}\quad
||S^*_{\mathbf{x}} D_w (y - S_{\mathbf{x}}f^*)||^2_{K^1} %& = & \sum_{i,j=1}^m  w_{x_i}w_{x_j}(y_i - f^*(x_i))(y_j - f^*(x_j)){K^1}(x_i,x_j)\\
& = &  \sum_{i,j=1}^m  w_{x_i}w_{x_j}\eta_{x_i}\eta_{x_j}{K^1}(x_i,x_j).
\end{eqnarray*}
Since we have assumed that the $\eta_x$ random variables are independent with zero mean, we have that
\begin{equation*}
\mathbb{E}(||S^*_{\mathbf{x}} D_w (y - S_{\mathbf{x}}f^*)||^2_{K^1})  =  \sum_{i=1}^m w_{x_i}^2 \sigma_{x_i}^2 {K^1}(x_i,x_i)
 \le  \sigma^2\kappa^2 \sum_{i=1}^m w^2_{x_i} = \sigma^2\kappa^2 ||\mathbf{w}||^2,
\end{equation*}
where the inequality follows from our assumption that  $\sigma^2:=\sup_{x\in X}\sigma_{x_i}^2 < \infty$. Then it follows that
\begin{equation*}
\mathbb{E}(||f_{\mathbf{z},\lambda} - f_{\mathbf{x},\lambda}||_{K^1}^2) \le ||(S^*_{\mathbf{x}} D_w S_{\mathbf{x}} + \lambda I)^{-1}||^2||\mathbf{w}||^2\sigma^2\kappa^2.
\end{equation*}
The operator $(S^*_{\mathbf{x}} D_w S_{\mathbf{x}} + \lambda I)^{-1}$ is estimated analagously to \cite[Proposition 1]{SmaZho05} to obtain
%To estimate the operator $(S^*_{\mathbf{x}} D_w S_{\mathbf{x}} + \lambda I)^{-1}$, let $f\in\mathcal{H}_{K^1}$ and $g = (S^*_{\mathbf{x}} D_w S_{\mathbf{x}} + \lambda I)^{-1}f$. Then
%\begin{equation*}
%(S^*_{\mathbf{x}} D_w S_{\mathbf{x}} + \lambda I)g  = f
%\end{equation*}
%and taking inner products with $g$ on both sides yields
%\begin{equation*}
%||D_w^{1/2}  S_{\mathbf{x}} g||_{\mathbb{R}^m}^2  + \lambda ||g||^2_{K^1}  =  \langle f, g \rangle_{K^1}
%\end{equation*}
%So then
%\begin{eqnarray*}
%\lambda ||g||^2_{K^1}  & \le & ||f||_{K^1} ||g||_{K^1} \\
% ||g||_{K^1} & \le & \frac{1}{\lambda}|| f ||_{K^1}
%\end{eqnarray*}
%Since this is true for all $f\in\mathcal{H}_{K^1}$, it follows that
\begin{equation}			\label{eq:SDS}
||(S^*_{\mathbf{x}} D_w S_{\mathbf{x}} + \lambda I)^{-1}|| \le \frac{1}{\lambda}.
\end{equation}
Then we have
\begin{equation*}
\mathbb{E}(||f_{\mathbf{z},\lambda} - f_{\mathbf{x},\lambda}||_{K^1}^2) \le \frac{||\mathbf{w}||^2\sigma^2\kappa^2}{\lambda^2}.
\end{equation*}
Finally, for $0 < \delta < 1$,  application of the Markov inequality to the random variable $||f_{\mathbf{z},\lambda} - f_{\mathbf{x},\lambda}||^2_{K^1}$ gives
\begin{equation*}
\mathbb{P}\left(||f_{\mathbf{z},\lambda} - f_{\mathbf{x},\lambda}||^2_{K^1} \ge  \frac{||\mathbf{w}||^2\sigma^2\kappa^2}{\lambda^2\delta}\right) \le \delta.
\end{equation*}
Combining the above together with \eqref{eq:supnorminclusion} proves the lemma.
\qquad\end{proof}

\subsection{Integration error}

To establish our estimate for the integration error we need to make additional assumptions on the choice of Borel measure $\rho$. Namely, we will require that  it is strongly continuous (c.f. \cite{Pag97}):

\begin{definition}
A Borel measure $\rho$ is strongly continuous if for all hyperplanes $H\subset\mathbb{R}^d$, we have $\rho(H)=0$.
\end{definition}

Note that this requirement implies that the boundaries of the Voronoi tessellation have $\rho$-measure zero. Lebesgue measure is still an example measure that satisfies all of our assumptions, recall Section \ref{sec:RKHS}.
An estimate for the integration error $||f_{\mathbf{x},\lambda} - f_\lambda||_{L^\infty(X)}$ is given in the following lemma.

\begin{lemma} 			\label{lem:integrationerror}
Let $\rho$ be a (finite) strictly positive, strongly continuous Borel measure on $X$.
%Also assume that ${K^1} \in C^{2s+\varepsilon}(X\times X)$ for $0<\varepsilon<2$. {\color{red} Is there a result for Wendland function to use here???}
For $\lambda > 0$, we have
\begin{equation}
||f_{\mathbf{x},\lambda} - f_\lambda||_{L^\infty(X)} \le \frac{C}{\lambda}\kappa ||f^* - f_\lambda||_{K^1} \, h_\mathbf{x} \rho(X).
\end{equation}
%where  $\kappa:= \sqrt{\sup_{x\in X}{K^1}(x,x)}$.
\end{lemma}
\begin{proof}
 It is shown in \cite{CucSma02} that the solution to \eqref{eq:flambda} for $\lambda > 0$ is given by
\begin{equation}
f_\lambda = (L_{K^1} + \lambda I)^{-1} L_{K^1} f^*,	\label{eq:flambdasol}
\end{equation}
where $L_{K^1}$ was defined in equation \eqref{eq:LK}.

Now we have
\begin{eqnarray*}
f_{\mathbf{x},\lambda} - f_\lambda & = & (S^*_{\mathbf{x}} D_w S_{\mathbf{x}} + \lambda I)^{-1} S^*_{\mathbf{x}} D_w S_{\mathbf{x}} f^* -  f_\lambda\\
%& = & (S^*_{\mathbf{x}} D_w S_{\mathbf{x}} + \lambda I)^{-1}\left\{ S^*_{\mathbf{x}} D_w S_{\mathbf{x}} f^* - S^*_{\mathbf{x}} D_w S_{\mathbf{x}} f_\lambda - \lambda f_\lambda\right\}\\
& = & (S^*_{\mathbf{x}} D_w S_{\mathbf{x}} + \lambda I)^{-1}\left\{ S^*_{\mathbf{x}} D_w S_{\mathbf{x}} (f^* -  f_\lambda) - (L_{K^1} f^* - L_{K^1} f_\lambda)\right\}\\
& = & (S^*_{\mathbf{x}} D_w S_{\mathbf{x}} + \lambda I)^{-1}\left\{ \sum_{i=1}^m w_{x_i} (f^*(x_i) -  f_\lambda(x_i))K^1_{x_i} - L_{K^1} (f^* -  f_\lambda)\right\}
\end{eqnarray*}
where the second equality follows from \eqref{eq:flambdasol} and the final equality follows from the definition of $S_{\mathbf{x}}$ and its adjoint.

Recall that we have chosen the weighting $\mathbf{w}$ to be equal to the $\rho$-volume of the Voronoi tessellation associated to the data sites $\mathbf{x}$ -- that is, $w_{x_i} = \rho(V_i(\mathbf{x}))$. Also, since $\rho$ is strongly continuous it holds that
\begin{equation*}
L_{K^1}(f^* - f_\lambda) %&= & \int_X {K^1}(x,y)(f^*(y) - f_\lambda(y))d\rho(y)\\
=\sum_{i=1}^m \int_{V_i(\mathbf{x})} {K^1}(x,y)(f^*(y) - f_\lambda(y))d\rho(y)
\end{equation*}
and so
\begin{eqnarray*}
\left|\left|\sum_{i=1}^m w_{x_i} (f^*(x_i) -  f_\lambda(x_i))K^1_{x_i} - L_{K^1} (f^* -  f_\lambda)\right|\right|_{L^\infty(X)}	\hspace{-7.3cm}\\
& = & \left|\left|\sum_{i=1}^m \left\{ w_{x_i} (f^*(x_i) -  f_\lambda(x_i))K^1_{x_i} - \int_{V_i(\mathbf{x})}{K_y^1} (f^*(y) - f_\lambda(y))d\rho(y)\right\}\right|\right|_{L^\infty(X)}\\
& \le & \sum_{i=1}^m \int_{V_i(\mathbf{x})} \left|\left|(f^*(x_i) -  f_\lambda(x_i))K^1_{x_i} - (f^*(y) - f_\lambda(y)) {K_y^1}\right|\right|_{L^\infty(X)} d\rho(y)\\
& \le & C\kappa ||f^* - f_\lambda||_{K^1} \sum_{i=1}^m \int_{V_i(\mathbf{x})} |x_i - y | d\rho(y)
\,\le\,  C\kappa ||f^* - f_\lambda||_{K^1} \, h_\mathbf{x} \rho(X)
\end{eqnarray*}
%The first inequality follows from a particular property of the Wendland kernel, that ${K^1}(x,y)\le {K^1}(x,x)$, for all $x\in X$.
The second equality follows from the fact that $(f^*-f_\lambda)$ and $K_x$ belong to $\mathcal{H}_{K^1}$, and are therefore bounded and Lipschitz on $X$ (cf. Corollary \ref{cor:smoothnessK1K2}).
This, together with \eqref{eq:SDS} proves the lemma.
\qquad\end{proof}

\subsection{Regularisation error}

For the regularisation error we  recall the following result from \cite[Lemma 3]{SmaZho07} (c.f. also \cite[Theorem 4]{SmaZho05}):

\begin{lemma}			\label{lem:regularisationerror}
Suppose that  $L_{K^1}^{-r} f^* \in {L}^2_\rho(X)$ for some  $\frac{1}{2} < r \le 1$. Then we have
%\begin{equation}
%||f_\lambda - f^*||^2_{\mathcal{L}^2_\rho} + \lambda||f^*||^2_{K^1} \le \lambda^{2r}||L_{K^1}^{-r}f^*||^2_{\mathcal{L}^2_\rho},\qquad 0<r \le \frac{1}{2}
%\end{equation}
%and
\begin{equation}				\label{eqn:estflambdafstar}
||f_\lambda - f^*||_{K^1} \le \lambda^{r-\frac{1}{2}} ||L_{K^1}^{-r} f^*||_{{L}^2_\rho(X)},\qquad \frac{1}{2} < r \le 1.
\end{equation}
%Furthermore, for $0 < r\le 1$ we have
%\begin{equation}
%||f_\lambda - f^*||_{\mathcal{L}^2_\rho} \le \lambda^r ||L_{K^1}^{-r} f^*||_{\mathcal{L}^2_\rho}.
%\end{equation}
\end{lemma}

\subsection{Estimate for $||f_{\mathbf{z},\lambda} - f^*||_{L^\infty(X)}$}

Altogether, from lemmas \ref{lem:sampleerror}, \ref{lem:integrationerror} and \ref{lem:regularisationerror} and equation \eqref{eq:supnorminclusion} we have with probability $1-\delta$
\begin{equation}		
||f_{\mathbf{z},\lambda} - f^*||_{L^\infty(X)} \le
\frac{||\mathbf{w}||_{\mathbb{R}^m}\sigma\kappa^2}{\lambda \sqrt{\delta}} +  \frac{C\kappa ||f^* - f_\lambda||_{K^1} \, h_\mathbf{x} \rho(X)}{\lambda}     + \kappa\lambda^{r-\frac{1}{2}}||L_{K^1}^{-r}f^*||_{{L}^2_\rho(X)}
\end{equation}
Applying equation \eqref{eqn:estflambdafstar} again yields
%we note that $[f^* - f_\lambda]_{\mathrm{Lip}} \le 4^s ||K^1||^{1/2}_{C^{2s}}||f^* - f_\lambda||_{K^1}$ and therefore
\begin{equation}
||f_{\mathbf{z},\lambda} - f^*||_{L^\infty(X)} \le C \left(
\frac{||\mathbf{w}||_{\mathbb{R}^m}}{\lambda \sqrt{\delta}} + \lambda^{r-\frac{3}{2}} h_\mathbf{x}+ \lambda^{r-\frac{1}{2}} \right),
\label{eqn:festimate}
\end{equation}
where the constant $C$ depends on $f^*$,  $d$, $\sigma$ and the choice of RKHS $\mathcal{H}_{K^1}$. Now it is clear that the above bound can be made arbitrarily small as $||\mathbf{w}||_{\mathbb{R}^m}$ and $h_{\mathbf{x}}$ tend to zero, if $\lambda$ also tends to zero at an appropriate rate.
With the choice of regularisation parameter
\begin{equation}					\label{eqn:lambdachoice}
\lambda = \left(\max\left\{||\mathbf{w}||_{\mathbb{R}^m}, h_{\mathbf{x}}^{\frac{2}{3-2r}}\right\}\right)^{\frac{2}{2r+1}},
\end{equation}
with probability $1-\delta$, we obtain the  estimate
\begin{equation}				\label{eqn:fzlambdaestnolambda}
||f_{\mathbf{z},\lambda} - f^*||_{L^\infty(X)} \le C \left(\max\left\{||\mathbf{w}||_{\mathbb{R}^m}/\sqrt{\delta}, h_{\mathbf{x}}\right\}\right)^{\frac{2r-1}{2r+1}}.
\end{equation}
%Minimizing the above for $\lambda >0$ we obtain the optimal asymptotic regularisation rate {\color{red} check this}
%\begin{equation}
%\lambda = \left( \frac{||\mathbf{w}||_{\mathbb{R}^m}\sigma}{4^s ||{K^1}||^{1/2}_{C^{2s}}\sqrt{\delta}(r-1) h_{\mathbf{x}}\rho(X) ||L_{K^1}^{-r}f^*||_{\mathcal{L}^2_{\rho}} } \right)^{1/r}.
%\end{equation}
%\hfill$\square$

%%%%%%%%%%%%%%%%%%%%%%%%%%%%%%%%%%%%%%%%%%%%%%%%%%%%%%%%%%%%%%%%%%%%%%%%%%%%%
%%%%%%%%%%%%%%%%%%%%%%%%%%%%%%%%%%%%%%%%%%%%%%%%%%%%%%%%%%%%%%%%%%%%%%%%%%%%%

 \section{Proof of  Theorem \ref{thm:mainresult}}		\label{sec:proofofmainthm}

Let $V\in C^{\nu_1}(A(\overline{x}),\mathbb{R})$ and $T\in C^{\nu_1}(A(\overline{x})\,\backslash\,\{\overline{x}\},\mathbb{R})$ be the Lyapunov functions for $f^*$ as defined in Theorems \ref{thm:VLyapunovconverse} and \ref{thm:TLyapunovconverse}. Then we have
\begin{equation}
L_{f^*}V(x) := \langle \nabla V(x), f^*(x) \rangle_{\mathbb{R}^d} = -p(x),\qquad \text{for all }x\in A(\overline{x}),
\end{equation}
with $p(x)$ also defined in Theorem \ref{thm:VLyapunovconverse}. Similarly,
\begin{eqnarray}
L_{f^*}T(x) &=& -\overline{c}\qquad \text{for all }x\in A(\overline{x})\backslash\{\overline{x}\},\\
T(x) & = & \xi_T(x), \qquad x\in\Gamma.			\label{eq:TxiT}
\end{eqnarray}
for $c>0$, where $\Gamma = \{x\in A(\overline{x})\backslash\{\overline{x}\} \mid h(x) = 0\}$ is a non-characteristic hypersurface according to Definition \ref{def:noncharhyp} (with $h\in C^{\nu_1}(\mathbb{R}^d,\mathbb{R})$), and $\xi_T\in C^{\nu_1}(\Gamma,\mathbb{R})$.

As stated in Theorem \ref{thm:VLyapunovconverse}, the Lyapunov function $V$ is uniquely defined up to a constant.
We will fix $V$ by setting $V(\overline{x})=0$. The Lyapunov function $T$ is uniquely defined according to the above properties.

%{\color{red} We will uniquely define $V$ by additionally requiring it to have minimal norm in the RKHS $\mathcal{H}_{K^2}$. This is a natural restriction in our context, as it can be shown that the solution provided by Algorithm \ref{alg:approxV} is precisely the minimal norm interpolant in $\mathcal{H}_{K^2}$ that satisfies the orbital derivative condition \eqref{eq:Veq} on the test points $\mathbf{q} = (q_i)_{i=1}^M$. \textbf{necessary??}}

The following Lemma provides an alternative characterisation of the Lyapunov function $V$, which will be useful later in the section.
\begin{lemma}						\label{lem:VxiV}
Let $V\in C^{\nu_1}(A(\overline{x}),\mathbb{R})$ be the uniquely defined Lyapunov function as above, and $\Gamma = \{x\in A(\overline{x})\backslash\{\overline{x}\} \mid h(x) = 0\}$ ($h\in C^{\nu_1}(\mathbb{R}^d,\mathbb{R})$) is a non-characteristic hypersurface according to Definition \ref{def:noncharhyp}. Define $\xi_V \in C^{\nu_1}(\Gamma, \mathbb{R})$ by
$\xi_V(x) := V(x)$ for $x\in\Gamma$. Also recall $\varphi_{f^*}(t,\cdot)$ denotes the flow operator of \eqref{eqn:dynsys}, and define the function $\theta_{f^*}\in C^{\nu_1}(A(\overline{x})\backslash\{\overline{x}\},\mathbb{R})$ by $\varphi_{f^*}(t,x)\in\Gamma \Leftrightarrow t = \theta_{f^*}(x)$.
Then
\begin{equation*}
V(x) = \xi_V(\varphi_{f^*}(\theta_{f^*}(x),x)) + \int_0^{\theta_{f^*}(x)} p(\varphi_{f^*}(\tau ,x))d\tau,\qquad x\in A(\overline{x})\backslash\{\overline{x}\}.
\end{equation*}
\end{lemma}
\begin{proof}
It is shown in \cite[Theorem 2.38]{Gie07:a} that the function $\theta_{f^*}$ is well-defined and belongs to $C^{\nu_1}(A(\overline{x})\backslash\{\overline{x}\},\mathbb{R})$. Also in \cite[Theorem 2.46]{Gie07:a} it is shown that
\begin{equation*}
V(x) = \int_0^\infty p(\varphi_{f^*}(\tau,x))d\tau
\end{equation*}
Let $y = \varphi_{f^*}(\theta_{f^*}(x),x) \in \Gamma$. Then, for $x\in A(\overline{x})\backslash\{\overline{x}\}$,
\begin{eqnarray*}
V(x) & = &  \int_{\theta_{f^*}(x)}^\infty p(\varphi_{f^*}(\tau,x))d\tau + \int_0^{\theta_{f^*}(x)} p(\varphi_{f^*}(\tau,x))d\tau\\
%& = & \int_{0}^\infty p(\varphi_{f^*}(\overline\tau + \theta_{f^*}(x),x))d\overline\tau +  \int_0^{\theta_{f^*}(x)} p(\varphi_{f^*}(\tau,x))d\tau \qquad\text{($\overline{\tau} = \tau - \theta_{f^*}(x)$)}\\
& = & \int_{0}^\infty p(\varphi_{f^*}(\overline\tau ,y))d\overline\tau + \int_0^{\theta_{f^*}(x)} p(\varphi_{f^*}(\tau,x))d\tau \qquad\text{(using $\overline{\tau} = \tau - \theta_{f^*}(x)$)}\\
& = & \xi_V(y) +  \int_0^{\theta_{f^*}(x)} p(\varphi_{f^*}(\tau,x))d\tau
\end{eqnarray*}
which proves the Lemma.
\qquad\end{proof}

%\begin{remark}						\label{rem:xiV}
%Note that once we have uniquely defined the Lyapunov function $V\in C^\nu(A(\overline{x}),\mathbb{R})$, this in turn defines a function $\xi_V \in C^\nu(\Gamma, \mathbb{R})$ by $\xi_V(x) := V(x)$ for $x\in\Gamma$, where $\Gamma$ is again the non-characteristic hypersurface defined above. The function $\xi_V$ will provide a useful alternative characterisation of the Lyapunov function $V$, see \textbf{???}.
%\end{remark}

\begin{remark}						\label{rem:xiT}
Similarly, the Lyapunov function $T$ has the representation
\begin{equation*}
T(x) = \xi_T(\varphi_{f^*}(\theta_{f^*}(x),x)) + \overline{c}\,\theta_{f^*}(x),\qquad x\in A(\overline{x})\backslash\{\overline{x}\},
\end{equation*}
with $\xi_T\in C^{\nu_1}(\Gamma,\mathbb{R})$ as above.
\end{remark}

We  aim to approximate the Lyapunov functions $V$ and $T$ in a compact subset of the basin of attraction $A(\overline{x})$. This subset is given by $\mathcal{D}:=\Omega\setminus B_\varepsilon(\overline{x})$ for a given $\varepsilon >0$, where  $\Omega$ is compact, cf. Theorem \ref{thm:mainresult}. See Figure \ref{fig:domains} for a sketch of these domains.

For the approximation of $V$, we define $\tilde\Omega_V := \{x\in A(\overline{x}) \mid V(x) \le R\}$  with $R>0$ large enough so that $\Omega \subset\tilde\Omega_V$. Similarly for $T$, we choose $\tilde\Omega_T := \{x\in A(\overline{x})\backslash\{\overline{x}\} \mid T(x) \le R\}$ with $R>0$ large enough so that $\Omega \subset\tilde\Omega_T$.

Recall that $\Gamma$ is a non-characteristic hypersurface for $f^*$ (and thus also for $f_{\mathbf{z},\lambda}$, if $f_{\mathbf{z},\lambda}$ and $f^*$ are sufficiently close in supremum norm). Additionally, we may define $\tilde\Gamma:= \varphi_{f^*}(T,\Gamma)$ with $T>0$ sufficiently large so that $\tilde\Gamma\subset B_\varepsilon(\overline{x})$.
% (here $\varphi$ denotes the flow of \eqref{eqn:dynsys}).
 Note that $\tilde\Gamma$ is  a non-characteristic hypersurface for $f^*$ (also for $f_{\mathbf{z},\lambda}$), defined by some $\tilde{h}\in C^{\nu_1}(\mathbb{R}^d,\mathbb{R})$. Then we define the Lipschitz domains ${\Omega}_V:=\tilde\Omega_V\cap\{x\in A(\overline{x})\mid \tilde{h}(x)\ge 0\}$ and ${\Omega}_T:=\tilde\Omega_T\cap\{x\in A(\overline{x})\mid \tilde{h}(x)\ge 0\}$ (cf. Theorem \ref{thm:GieWen}), and note that $\mathcal{D}\subset\Omega_V$ and $\mathcal{D}\subset\Omega_T$. Also note that all orbits (for $f^*$ and $f_{\mathbf{z},\lambda}$) enter and exit ${\Omega}_V$ (and ${\Omega}_T$) only once.

%in the following theorem (see \cite[Corollary 2.43, Proposition 2.44 and Theorem 2.46]{Gie07:a}).
%
%\begin{theorem}
%Let $\overline{x}$ be an equilibrium of $\dot{x} = f^*(x)$, $f\in C^\nu(\mathbb{R}^d,\mathbb{R}^d)$, $\nu \ge 1$, and let the maximal real part of all eigenvalues of $Df(\overline{x})$ be negative. Let $f$ be bounded in $A(\overline{x})$, and let $L$ be one of the Lyapunov functions defined in Theorems \ref{thm:VLyapunovconverse} and \ref{thm:TLyapunovconverse}; that is, $L:=V$ or $L:=T$.
%
%Then for all $r>0$ the set
%\begin{equation}
%\Omega := \{x\in A(\overline{x})\backslash\{\overline{x}\} \mid L(x) \le r\} \cup \{\overline{x}\}	\label{eq:Omega}
%\end{equation}
%is compact. Moreover, there is a $C^\nu$ diffeomorphism
%\begin{equation*}
%\phi \in C^\nu(S^{d-1}, \{ x \in A(\overline{x}) \mid L(x) = r\}),
%\end{equation*}
%where $S^{d-1} = \{x\in\mathbb{R}^d \mid ||x||_2 = 1\}$. For $L=T$ (cf. Theorem \ref{thm:TLyapunovconverse}) we have $\lim_{x\rightarrow \overline{x}} L(x) = -\infty$.
%\end{theorem}
%
%In the case $L=T$, it is first necessary to link $L$ to a local Lyapunov function for the above theorem to hold, see \cite{Gie07}.

Our algorithm detailed in \S\ref{sec:algorithm} computes the generalised interpolant approximations $\hat{V}$ and $\hat{T}$ corresponding to the vector field approximation $f_{\mathbf{z},\lambda}$ (as in Theorem \ref{thm:GieWen} with $g = f_{\mathbf{z},\lambda}$).
%We show in \S\ref{sec:errorf} that the error made in approximating $f^*$ by  $f_{\mathbf{z},\lambda}$ is small in the supremum norm.
%However, this does not guarantee that stability properties of the equilibrium $\overline{x}$ will be preserved. This is why our estimates are not valid in a neighbourhood of the equilibrium.

Note that for    $\max_k\left(||f^k_{\mathbf{z},\lambda}-f^{*,k}||_{L^\infty(X)}\right)$ sufficiently small (recall the superscript $k$ denotes the $k$-th component), $f_{\mathbf{z},\lambda}$ does not have any equilibria in $\Omega_V$ (resp. $\Omega_T$). Similarly,  $\Gamma, \tilde{\Gamma}$ are both non-characteristic hypersurfaces for $f_{\mathbf{z},\lambda}$, and  all trajectories in  $\Omega_V$ (resp. $\Omega_T$) eventually enter (and stay in) the region defined by $\{x\in A(\overline{x}) \mid \tilde{h}(x) <0\}$.
%These conditions can be guaranteed if $||f_{\mathbf{z},\lambda}-f^*||_\infty$ is sufficiently small.

In fact, for $||\mathbf{w}||_{\mathbb{R}^m}$ and $h_{\mathbf{x}}$ sufficiently small, $f_{\mathbf{z},\lambda}$ and $f^*$ are even close in a $C^{\nu_2}$ sense, as we will show in the following Lemmas \ref{lem:fzlambdafstarbounded} and \ref{lem:fzlambdafstarepsilon}.

\begin{lemma}			\label{lem:fzlambdafstarbounded}
For $\lambda >0$, for every $0<\delta<1$, with probability $1-\delta$, we have
\begin{equation}		\label{eq:fzlambdafstarbounded}
||f^k_{\mathbf{z},\lambda} - f^{*,k}||_{K^1} \le \frac{||\mathbf{w}||_{\mathbb{R}^m}\sigma\kappa}{\lambda\sqrt{\delta}} + 2 ||f^{*,k}||_{K^1},\qquad k=1,\ldots,d.
\end{equation}
%where $\kappa:= \sqrt{\sup_{x\in X} K^1(x,x)}$.
\end{lemma}
\begin{proof}
From Lemma \ref{lem:fzlambda} and equation \eqref{eq:fxlambda}, we see that for each $k=1,\ldots,d$:
\begin{equation}			\label{eq:fkxlambda}
f^k_{\mathbf{x},\lambda}=\arg\min_{g\in\mathcal{H}_{K^1}}\left\{ \sum_{i=1}^m w_{x_i}(g(x_i)-f^{*,k}(x_i))^2 + \lambda||g||^2_{K^1}\right\}.
\end{equation}
We will show that $||f^k_{\mathbf{x},\lambda}||_{K^1}$ is bounded independently of $(\mathbf{x},\lambda)$. To see this, first note that due to the choice of the positive definite Wendland function kernel $K^1$, that it is always possible to find a norm-minimal function $g^k_0\in\mathcal{H}_{K^1}$ that interpolates the data. That is, $g^k_0$ is the solution to the problem
\begin{equation*}
\min_{g\in\mathcal{H}_{K^1}}\left\{||g||_{K^1}~:~ g(x_i) = f^{*,k}(x_i),~i=1,\ldots,m\right\}.
\end{equation*}
Therefore we have that $||g^k_0||_{K^1} \le ||f^{*,k}||_{K^1}$. Now from \eqref{eq:fkxlambda} we have the following:
\begin{eqnarray*}
\lambda || f^k_{\mathbf{x},\lambda} ||^2_{K^1} & \le & \sum_{i=1}^m w_{x_i}( f^k_{\mathbf{x},\lambda}(x_i)-f^{*,k}(x_i))^2 + \lambda|| f^k_{\mathbf{x},\lambda}||^2_{K^1}\\
& \le & \sum_{i=1}^m w_{x_i}(g^k_0(x_i)-f^{*,k}(x_i))^2 + \lambda||g^k_0||^2_{K^1}\\
& = & \lambda ||g^k_0||^2_{K^1}\\
& \le & \lambda ||f^{*,k}||_{K^1}^2
\end{eqnarray*}
So then $|| f^k_{\mathbf{x},\lambda} - f^{*,k} ||_{K^1} \le || f^k_{\mathbf{x},\lambda}||_{K^1} + ||f^{*,k}||_{K^1} \le 2 ||f^{*,k}||_{K^1}$.

In Lemma \ref{lem:sampleerror} we have provided a bound for $|| f^k_{\mathbf{z},\lambda} - f^k_{\mathbf{x},\lambda}||_{K^1}$ for a given probability $1 - \delta$. Then together we find with probability $1 - \delta$,
\begin{eqnarray*}
||f^k_{\mathbf{z},\lambda} - f^{*,k}||_{K^1} & \le & ||f^k_{\mathbf{z},\lambda} - f^k_{\mathbf{x},\lambda}||_{K^1} + ||f^k_{\mathbf{x},\lambda} - f^{*,k}||_{K^1} \\
& \le & \frac{||\mathbf{w}||_{\mathbb{R}^m}\sigma\kappa}{\lambda\sqrt{\delta}} + 2 ||f^{*,k}||_{K^1}
\end{eqnarray*}
which proves the Lemma.
\qquad\end{proof}

We will need the following convergence result in Lemma \ref{lem:fzlambdafstarepsilon}.% It is a generalisation of Theorem 2.5 that appears in \cite{GieWen07}.

\begin{theorem}[\cite{WenRie05}]		\label{thm:WenRie}
Suppose $X\subseteq\mathbb{R}^d$ is bounded and satisfies an interior cone condition with radius $r$ and angle $\theta$. Let $\tilde\tau$ be a positive integer, $0< s \le 1$, $1\le p < \infty$, $1 \le q \le \infty$ and let $m\in\mathbb{N}_0$ satisfy $\tilde\tau > m+d/p$, or, if $p=1$, $\tilde\tau\ge m+d$. Then there exists a constant $C>0$ depending only on $\tilde\tau,d,p,q,m,\theta$  such that every discrete set $\Pi\subseteq\Omega$ with mesh norm $h_{\Pi}$ sufficiently small, and every $u\in W_p^{\tilde\tau+s}(X)$ the estimate
\begin{equation}
|u|_{W_q^m(X)} \le C\left( h_\Pi^{\tilde\tau+s-m-d(1/p - 1/q)_+} |u|_{W_p^{\tilde\tau+s}(X)} + h_{\Pi}^{-m}||u|\Pi||_{l^\infty(\Pi)}\right)
\end{equation}
is satisfied. Here, $(x)_+ = \max\{x,0\}$, and we use the notation $u|\Pi$ to denote the restriction of $u$ to the set $\Pi$.
%
%The quantities $Q(k,\theta)$ and $r$ in the above estimate are determined by the domain $\Omega$ and are given by:
%\begin{eqnarray*}
%\vartheta&:=& 2 \arcsin\left(\frac{\sin\theta}{4(1 + \sin\theta)}\right),\\
%Q(k,\theta) &:= & \frac{\sin\theta\sin\vartheta}{8k^2(1+\sin\theta)(1 + \sin\vartheta)},\\
%R &:=& Q(k,\theta)^{-1} h,\\
%r &:=& \frac{\sin\theta}{2(1+\sin\theta)}R.
%\end{eqnarray*}
\end{theorem}

\begin{lemma} 			\label{lem:fzlambdafstarepsilon}
Let $\varepsilon >0$ be arbitrarily small. For every $0 < \delta < 1$, there exists $\iota >0$ such that when
$||\mathbf{w}||_{\mathbb{R}^m}, h_{\mathbf{x}} <\iota$, and $\lambda>0$ chosen according to \eqref{eqn:lambdachoice}, the following estimate holds with probability $1-\delta$:
\begin{eqnarray*}
||f^k_{\mathbf{z},\lambda} - f^{*,k}||_{C^{\nu_2}(X)}& <& \varepsilon,\qquad k=1,\ldots,d,
\end{eqnarray*}
where $\nu_2:= \lceil\tau_2\rceil = k_1-(d+3)/2$  ($d$  odd), and $\nu_2:= \lceil\tau_2\rceil = k_1-(d+4)/2$ ($d$ even).
 \end{lemma}
\begin{proof}
Note that since $X$ has a $C^1$ boundary, it satisfies the interior cone condition from Theorem~\ref{thm:WenRie} (see \cite[Definition 3.6]{Wen05}).
Let $j\in\mathbb{N}_0$ be such that $j\le k_1 - 1$.
Recall that $\tau_1 = k_1 + (d+1)/2$ with $\lceil\tau_1\rceil = \nu_1$. Then when $d$ is odd, we have $j < \tau_1 - 1 - d/2$. When $d$ is even, we have $\lfloor\tau_1\rfloor - d/2 = k_1$ and so $j < \lfloor\tau_1\rfloor - d/2$.
%Recall from \S\ref{sec:SobolevspaceRKHS} that our smoothness assumptions in Theorem \ref{thm:mainresult} imply that when $d$ is odd, we have $\tau\ge d+2$. In this case, let $j\in\mathbb{N}_0$ be such that  $j\le (d+3)/2$. Then $j < \lfloor\tau\rfloor - d/2$, where $\tau$ is as in Theorem \ref{thm:mainresult}. Similarly, when $d$ is even, we have $\tau \ge d + 9/2$, and in this case let $j\in\mathbb{N}_0$ be such that $j \le (d+6)/2$. Again we have $j < \lfloor\tau\rfloor - d/2$.

Using Theorem \ref{thm:WenRie}, and the fact that $\mathcal{H}_{K^1}$ and $W_2^{\tau_1}$ are norm-equivalent, we have the following estimate:
\begin{equation*}
|f^k_{\mathbf{z},\lambda} - f^{*,k}|_{W_2^j(X)} \le C\left( h_{\Pi}^{\tau_1-j} |f^k_{\mathbf{z},\lambda} - f^{*,k}|_{W_2^{\tau_1}(X)} + h_{\Pi}^{-j}||f^k_{\mathbf{z},\lambda} - f^{*,k}||_{L^\infty(X)}\right).
\end{equation*}
Note that we have replaced $||f^k_{\mathbf{z},\lambda} - f^{*,k}|\Pi||_{l^\infty(\Pi)}$ in Theorem \ref{thm:WenRie} with $||f^k_{\mathbf{z},\lambda} - f^{*,k}||_{L^\infty(X)}$. Then the discrete set $\Pi$ from Theorem \ref{thm:WenRie} may be taken to be any discrete set in $X$, and so the fill distance $h_{\Pi}$ in the above can be taken to be arbitrarily small.

Now, from \eqref{eqn:fzlambdaestnolambda} we see that $||f^k_{\mathbf{z},\lambda} - f^{*,k}||_{L^\infty(X)}$
%we show that there exists some $r$ with $\frac{1}{2} < r \le 1$
%such that the quantity $||f^k_{\mathbf{z},\lambda} - f^{*,k}||_{L^\infty(X)}$ can be estimated as
%\begin{equation*}
%||f_{\mathbf{z},\lambda} - f^*||_{L^\infty(X)} \le
%\frac{||\mathbf{w}||_{\mathbb{R}^m}\sigma\kappa^2}{\lambda \sqrt{\delta}} + \lambda^{r-1}4^s ||K^1||^{1/2}_{C^{2s}}||L_{K^1}^{-r}f^*||_{\mathcal{L}^2_{\rho}} \kappa^2 h_\mathbf{x} \rho(X) + \kappa\lambda^{r-\frac{1}{2}}||L_{K^1}^{-r}f^*||_{L^\infty(X)}	
%\end{equation*}
 can be made arbitrarily small for small $||\mathbf{w}||_{\mathbb{R}^m}$ and $h_{\mathbf{x}}$, and suitably chosen $\lambda>0$ as in \eqref{eqn:lambdachoice}. The estimate \eqref{eqn:fzlambdaestnolambda}  holds with probability $1-\delta$ (for $0<\delta<1$), where $\delta$ here is the same as in Lemma \ref{lem:fzlambdafstarbounded}, as the estimate depends on the same probabilistic inequality for $||f^k_{\mathbf{z},\lambda} - f^k_{\mathbf{x},\lambda}||_{K^1}$ (cf. \eqref{eq:sampleerror}).
Then we have from \eqref{eq:fzlambdafstarbounded} (and using again the norm-equivalance of $\mathcal{H}_{K^1}$ and $W_2^{\tau_1}$), that $ |f^k_{\mathbf{z},\lambda} - f^{*,k}|_{W_2^{\tau_1}(X)}$ is bounded, say $ |f^k_{\mathbf{z},\lambda} - f^{*,k}|_{W_2^{\tau_1}(X)} \le \tilde{C}/k_1$.

Now, given $\tilde\varepsilon>0$, setting
\begin{equation*}
||f^k_{\mathbf{z},\lambda} - f^{*,k}||_{L^\infty(X)}  <  \frac{\tilde{C}}{k_1}\left(\frac{\tilde\varepsilon}{2C\tilde{C}}\right)^{\frac{\tau_1}{\tau_1 - j}}
\qquad\text{and}\qquad h_{\Pi}  <  \left(\frac{\tilde\varepsilon}{2 C \tilde{C}}\right)^{\frac{1}{\tau_1 - j}},
\end{equation*}
we have that
%\begin{eqnarray*}
%|f^k_{\mathbf{z},\lambda} - f^{*,k}|_{W_2^j(X)} & \le & \frac{C}{k_1}(h_{\mathcal{D}}^{\tau_1 - j} \tilde{C} + h_{\mathcal{D}}^{\tau_1 - j} \tilde{C})\\
%& < & \frac{\varepsilon}{2} + \frac{\varepsilon}{2} = \varepsilon
%\end{eqnarray*}
\begin{equation*}
|f^k_{\mathbf{z},\lambda} - f^{*,k}|_{W_2^j(X)}  \le \frac{\tilde\varepsilon}{k_1}.
\end{equation*}
Now using $||f^k_{\mathbf{z},\lambda} - f^{*,k}||_{W_2^{k_1-1}(X)} = \sum_{j=0}^{k_1-1} |f^k_{\mathbf{z},\lambda} - f^{*,k}|_{W_2^j(X)}$ gives the bound $||f^k_{\mathbf{z},\lambda} - f^{*,k}||_{W_2^{k_1-1}(X)} \allowbreak  \le \tilde\varepsilon.$
Then by Lemma \ref{lem:genSobolev} (also using arguments similar to  Corollary \ref{cor:smoothnessK1K2} since $X$ is closed), we have that  $||f^k_{\mathbf{z},\lambda} - f^{*,k}||_{C^{\nu_2}(X)} \le C ||f^k_{\mathbf{z},\lambda} - f^{*,k}||_{W_2^{k_1-1}(X)}$ and setting $\varepsilon = C\tilde\varepsilon$ proves the Lemma.
\qquad\end{proof}

% In fact, we have that $\mathcal{H}_{K^1}\subset C^{\nu_2}(X)$
%  The following Corollary follows easily from Lemma \ref{lem:fzlambdafstarepsilon} and \eqref{eq:genSobolev}.

% \begin{corollary}				\label{cor:fzlambdafstarepsilon}
%Let $\varepsilon >0$ be arbitrarily small.
%For every $0 < \delta < 1$, there exists $\iota >0$ such that when
%$||\mathbf{w}||_{\mathbb{R}^m}, h_{\mathbf{x}} <\iota$, and $\lambda>0$ suitably chosen according to \eqref{eqn:lambdachoice}, the following estimate holds with probability $1-\delta$:
%\begin{equation}
%||f^k_{\mathbf{z},\lambda} - f^{*,k}||_{C^{\nu_2}(X)} < \varepsilon,\qquad k=1,\ldots,d,
%\end{equation}
%%where $\gamma = \lfloor\frac{d}{2}\rfloor + 1 - \frac{d}{2}$ if $d$ is odd, and $\gamma$ is any element in $(0,1)$ if $d$ is even.
%where  $\nu_2 := \lceil\tau_2\rceil =  k_1-1$.
% \end{corollary}
% \begin{proof}
%From  Corollary \ref{cor:smoothnessK1K2}, we have that $\mathcal{H}_{K^1}\subset C^{\nu_2}(X)$. Then $||f^k_{\mathbf{z},\lambda} - f^{*,k}||_{W_2^{\nu_2}(X)} = ||f^k_{\mathbf{z},\lambda} - f^{*,k}||_{C^{\nu_2}(X)}$ and the Corollary follows immediately from Lemma \ref{lem:fzlambdafstarepsilon}.
% \qquad\end{proof}

It follows that for sufficiently small $||\mathbf{w}||_{\mathbb{R}^m}$ and $h_{\mathbf{x}}$, that (with probability $1-\delta$) $f_{\mathbf{z},\lambda}$ will have an equilibrium close to $\overline{x}$ which is also exponentially asymptotically stable. In addition, the non-characteristic hypersurfaces $\Gamma$ and $\tilde{\Gamma}$ for $f^*$ will also be non-characteristic hypersurfaces for $f_{\mathbf{z},\lambda}$ (as will the level sets $\{x\in A(\overline{x})\mid V(x)=R\}$, resp. $\{x\in A(\overline{x})\mid T(x)=R\}$). In this case we can define the following `Lyapunov-type' functions for $f_{\mathbf{z},\lambda}$.

%Now, for a given $f_{\mathbf{z},\lambda}$, define the operator ${L_{\mathbf{z},\lambda}}$ by
%\begin{equation*}
%{L}_{\mathbf{z},\lambda}(\cdot) := \langle \nabla (\cdot), f_{\mathbf{z},\lambda}(\cdot) \rangle_{\mathbb{R}^d}.
%\end{equation*}

\begin{definition}				\label{def:VzlambdaTzlambda}
Let $\varphi_{\mathbf{z},\lambda}(t,\cdot)$ denote the flow operator for the system $\dot{x} = f_{\mathbf{z},\lambda}(x)$. For $||\mathbf{w}||_{\mathbb{R}^m}$ and $h_{\mathbf{x}}$ sufficiently small, $\lambda>0$ chosen according to \eqref{eqn:lambdachoice}, and $0<\delta<1$,  $\Gamma$ will be a non-characteristic hypersurface for $f_{\mathbf{z},\lambda}$ with probability $1-\delta$.
Then the function $\theta_{{\mathbf{z},\lambda}}:\Omega_V\rightarrow \mathbb{R}$
given by $\varphi_{
{\mathbf{z},\lambda}}(t,x) \in \Gamma \Leftrightarrow t = \theta_{{\mathbf{z},\lambda}}(x)$ is well-defined. By a slight abuse of notation we will also similarly define $\theta_{{\mathbf{z},\lambda}}: \Omega_T\rightarrow\mathbb{R}$.

We define the functions $V_{\mathbf{z},\lambda}: \Omega_V\rightarrow\mathbb{R}$ and $T_{\mathbf{z},\lambda}: \Omega_T\rightarrow\mathbb{R}$
%$V_{\mathbf{z},\lambda}\in C^{\nu_2}(\Omega_V,\mathbb{R})$ and $T_{\mathbf{z},\lambda}\in C^{\nu_2}(\Omega_T,\mathbb{R})$
by
\begin{eqnarray}
V_{\mathbf{z},\lambda}(x) & = & \xi_V(\varphi_{{\mathbf{z},\lambda}}(\theta_{{\mathbf{z},\lambda}}(x),x)) + \int_0^{\theta_{{\mathbf{z},\lambda}}(x)} p(\varphi_{{\mathbf{z},\lambda}}(\tau ,x))d\tau,\qquad x\in \Omega_V,\label{eqn:Vzlambda}\\
T_{\mathbf{z},\lambda}(x) & = & \xi_T(\varphi_{{\mathbf{z},\lambda}}(\theta_{{\mathbf{z},\lambda}}(x),x)) + \overline{c}\, \theta_{{\mathbf{z},\lambda}}(x),\qquad x\in \Omega_T,			\label{eqn:Tzlambda}
\end{eqnarray}
where $\xi_V\in C^{\nu_1}(\Gamma,\mathbb{R})$ and $\xi_T\in C^{\nu_1}(\Gamma,\mathbb{R})$ are as in Lemma \ref{lem:VxiV} and equation \eqref{eq:TxiT} respectively.
%(Recall that $f_{\mathbf{z},\lambda}\in C^{\nu_2}(X)$ and so $\varphi_{{\mathbf{z},\lambda}}$ is also $C^{\nu_2}$ in both arguments.)
\end{definition}

%Therefore the Lyapunov function converse Theorems \ref{thm:VLyapunovconverse} and \ref{thm:TLyapunovconverse} apply, and we find the Lyapunov functions  $V_{\mathbf{z},\lambda}\in C^1(\Omega_V,\mathbb{R})$, and $T_{\mathbf{z},\lambda}\in C^1(\Omega_T,\mathbb{R})$ corresponding to the non-characteristic hypersurface $\Gamma$ and function $\xi\in C^1(\Gamma,\mathbb{R})$.

%We would like to show that the functions $V_{\mathbf{z},\lambda}$ and $T_{\mathbf{z},\lambda}$ are $C^{\nu_2}$-close to the Lyapunov functions $V$ and $T$ in the domains $\Omega_V$ and $\Omega_T$ respectively, for $f_{\mathbf{z},\lambda}$ sufficiently $C^{\nu_2}$-close to $f^*$.
%
%The following Lemma shows that the norm $||V_{\mathbf{z},\lambda}||_{C^{\nu_2}}$ is bounded for $||\mathbf{w}||_{\mathbb{R}^m}$ and $h_{\mathbf{x}}$ sufficiently small.

%{\color{red} The following can readily be verified directly from \eqref{eqn:Vzlambda} and \eqref{eqn:Tzlambda}:
%\begin{eqnarray*}
%\langle \nabla V_{\mathbf{z},\lambda}(x),f_{\mathbf{z},\lambda}(x)\rangle_{\mathbb{R}^d}& =& -p(x),\qquad x\in\Omega_V,\\
%\langle \nabla T_{\mathbf{z},\lambda}(x),f_{\mathbf{z},\lambda}(x)\rangle_{\mathbb{R}^d}& = &-\overline{c},\qquad x\in\Omega_T.
%\end{eqnarray*}}

In the proof of the following Lemma we show that in fact $\theta_{\mathbf{z},\lambda}  \in C^{\nu_2}(\Omega_V\cup\Omega_T,\mathbb{R})$, $V_{\mathbf{z},\lambda}\in C^{\nu_2}(\Omega_V,\mathbb{R})$ and $T_{\mathbf{z},\lambda}\in C^{\nu_2}(\Omega_T,\mathbb{R})$ .
\begin{lemma}					\label{lem:Vzlambdabounded}
For every $\varepsilon_1 >0$, and every $0 < \delta < 1$, there is  $\varepsilon_2 >0$ such that if $\max\left\{||\mathbf{w}||_{\mathbb{R}^m}, h_{\mathbf{x}}\right\} < \varepsilon_2$ and  $\lambda>0$ is chosen according to \eqref{eqn:lambdachoice}, then we have with probability $1 - \delta$:
\begin{eqnarray*}
||V_{\mathbf{z},\lambda} - V ||_{C^{\nu_2}(\Omega_V)} < \varepsilon_1,\\
||T_{\mathbf{z},\lambda} - T ||_{C^{\nu_2}(\Omega_T)} < \varepsilon_1.
\end{eqnarray*}
\end{lemma}
\begin{proof} We will prove the result for $V_{\mathbf{z},\lambda}$, as the proof for $T_{\mathbf{z},\lambda}$ is similar.
We will show that $V_{\mathbf{z},\lambda}\in C^{\nu_2}(\Omega_V,\mathbb{R})$, and $||V_{\mathbf{z},\lambda} - V ||_{C^{\nu_2}(\Omega_V)}$ can be made arbitrarily small as $||f_{\mathbf{z},\lambda} - f^* ||_{C^{\nu_2}(\Omega_V)}\rightarrow 0$. Then the result will follow from Lemma \ref{lem:fzlambdafstarepsilon}. The proof follows the ideas contained in \cite[Theorem 2.38]{Gie07:a}.

We consider a one-parameter family of vector fields ${f}(\cdot,\mu)$, $\mu\in\mathbb{R}$, in the $C^{\nu_2}$ topology such that ${f}(\cdot,0) = f^*$. Let $\varphi(t,\cdot,\mu)$ denote the corresponding one-parameter family of flow operators and note that $\varphi$ is $C^{\nu_2}$ in each of its arguments. For $\varepsilon>0$ sufficiently small, $|\mu|<\varepsilon$,  $\Gamma$ is a non-characteristic hypersurface for each $f(\cdot,\mu)$, and  all orbits of $f(\cdot,\mu)$ in $\Omega_V$ enter and exit $\Omega_V$ precisely once. Then we define the one-parameter family of functions $\theta(\cdot,\mu) :\Omega_V\rightarrow\mathbb{R}$ by $\varphi(t,x,\mu) \in \Gamma \Leftrightarrow t = \theta(x,\mu)$. We show that $\theta\in C^{\nu_2}(\Omega_V\times[-\varepsilon,\varepsilon],\mathbb{R})$ by the implicit function theorem. Note that $\theta$ is the solution $t$ to
\begin{equation}
F(x,t,\mu) := h(\varphi(t,x,\mu))	=0		\label{eq:implicittheta}
\end{equation}
where $h$ is as in Definition \ref{def:noncharhyp}. Let $(t^*,x^*,\mu^*)$ be a solution to \eqref{eq:implicittheta}. Then we have $ \frac{d}{dt}F(t^*,x^*,\mu^*) <0$
by Definition \ref{def:noncharhyp}. But since $h\in C^{\nu_1}(\mathbb{R}^d,\mathbb{R})$ and $\varphi$ is a $C^{\nu_2}$ function in $(x,t,\mu)$, we have that $\theta\in C^{\nu_2}(\Omega_V\times[-\varepsilon,\varepsilon],\mathbb{R})$ by the implicit function theorem.

For each $\mu$, define
\begin{equation}					\label{eqn:Vtilde}
\tilde{V}(x,\mu)  =  \xi_V(\varphi(\theta(x,\mu),x,\mu)) + \int_0^{\theta(x,\mu)} p(\varphi(\tau ,x,\mu))d\tau,\qquad x\in \Omega_V.
\end{equation}
Then it follows that $\tilde{V}\in C^{\nu_2}(\Omega_V\times [-\varepsilon,\varepsilon],\mathbb{R})$. It may also readily be verified that
\begin{equation*}
\langle \nabla \tilde{V}(x,\mu),f(x,\mu)\rangle_{\mathbb{R}^d} = -p(x),\qquad x\in\Omega_V.
\end{equation*}
Note that $\tilde{V}(x,0) = V(x)$ by Lemma \ref{lem:VxiV}. Now it is clear by \eqref{eqn:Vtilde} that $||\tilde{V}(\cdot,\mu) - V ||_{C^{\nu_2}}\rightarrow 0$ as $\mu\rightarrow 0$. But since $f(\cdot,\mu)$ is any one parameter family in the $C^{\nu_2}$ topology with $f(\cdot,0) = f^*$, we use Lemma  \ref{lem:fzlambdafstarepsilon} (and $\Omega_V\subset X$) to deduce that $V_{\mathbf{z},\lambda}\in C^{\nu_2}(\Omega_V,\mathbb{R})$, and $||V_{\mathbf{z},\lambda} - V ||_{C^{\nu_2}}\rightarrow 0$ as $||f_{\mathbf{z},\lambda} - f^* ||_{C^{\nu_2}}\rightarrow 0$ for $||\mathbf{w}||_{\mathbb{R}^m}$ and $h_{\mathbf{x}}$ sufficiently small, and $\lambda>0$ chosen according to \eqref{eqn:lambdachoice}.
\qquad\end{proof}

\begin{remark}				\label{rem:orbderVTzlambda}
It follows from the proof of Lemma \ref{lem:Vzlambdabounded} and from \eqref{eqn:Vzlambda} and \eqref{eqn:Tzlambda} that provided $\theta_{\mathbf{z},\lambda}$ is well defined (which is guaranteed with probability $1-\delta$), we have:
\begin{eqnarray}
\langle \nabla V_{\mathbf{z},\lambda}(x),f_{\mathbf{z},\lambda}(x)\rangle_{\mathbb{R}^d}& =& -p(x),\qquad x\in\Omega_V,	\label{eqn:Vzlambdaorbder}\\
\langle \nabla T_{\mathbf{z},\lambda}(x),f_{\mathbf{z},\lambda}(x)\rangle_{\mathbb{R}^d}& = &-\overline{c},\qquad x\in\Omega_T.	\label{eqn:Tzlambdaorbder}
\end{eqnarray}
\end{remark}
We now define a pairwise distinct, discrete set of points $\mathbf{q}:=(q_i)_{i=1}^M \subset \Omega_V$ (resp. $\Omega_T$). Note that these points need not be the same as $\mathbf{x}$. Let $h_{\mathbf{q}}$ be the fill distance of $\mathbf{q}$ in $\Omega_V$ (resp. $\Omega_T$). We compute our approximations $\hat{V}$ and $\hat{T}$ according to our algorithm given in \S\ref{sec:algorithm}.
%These approximations are essentially reconstructing the `Lyapunov functions' $V_{\mathbf{z},\lambda}$ and $T_{\mathbf{z},\lambda}$ according to the vector field $f_{\mathbf{z},\lambda}$
%The approximation is the minimal-norm interpolant $\hat{V}_{\mathbf{z},\lambda}\in\mathcal{H}_K$ that satisfies $L_{\mathbf{z},\lambda} \hat{V}_{\mathbf{z},\lambda} = -p(x)$ on the grid points $\mathbf{t}$. The solution to this interpolation problem is given by .....
We have (for $\hat{V}$, the arguments for $\hat{T}$ are similar)
\begin{eqnarray*}
\langle \nabla\hat{V}, f^*\rangle_{\mathbb{R}^d} & = &  \langle \nabla\hat{V}, f_{\mathbf{z},\lambda}  -f_{\mathbf{z},\lambda} +f^*\rangle_{\mathbb{R}^d}\\
\Rightarrow\langle \nabla\hat{V}, f^*\rangle_{\mathbb{R}^d} + p(\cdot) & = &  \langle \nabla\hat{V}, f_{\mathbf{z},\lambda}\rangle_{\mathbb{R}^d}  - \langle\nabla\hat{V},f_{\mathbf{z},\lambda} -f^*\rangle_{\mathbb{R}^d} + p(\cdot).
\end{eqnarray*}
Then we have, for $x\in\mathcal{D}$,
\begin{eqnarray*}
\langle \nabla\hat{V}(x), f^*(x)\rangle_{\mathbb{R}^d} + p(x) & \le & \langle \nabla\hat{V}(x), f_{\mathbf{z},\lambda} (x)\rangle_{\mathbb{R}^d} + p(x)\nonumber\\
&&  + \tilde{C}_2\max_k\left(||f^k_{\mathbf{z},\lambda} - f^{*,k}||_{L^\infty(\mathcal{D})} . ||(\nabla \hat{V})^k||_{L^\infty(\mathcal{D})}\right)  \nonumber\\
& \le & \tilde{C}_1 h_{\mathbf{q}}^{k_2-\frac{1}{2}} ||{V}_{\mathbf{z},\lambda}||_{W^{\tau_2}_2(\Omega_V)} \nonumber\\
& &  +\tilde{C}_2\max_k\left(||f^k_{\mathbf{z},\lambda} - f^{*,k}||_{L^\infty(X)} . ||(\nabla \hat{V})^k||_{L^\infty(\Omega_V)}\right),
\end{eqnarray*}
where recall that $\tau_2 := k_2 + (d+1)/2$ is the degree of the Sobolev RKHS $\mathcal{H}_{K^2}$. The last inequality above follows from Remark \ref{rem:orbderVTzlambda}, Theorem \ref{thm:GieWen}
%(choosing $\Omega_V$ large enough so that $\Omega_{V_{\mathbf{z},\lambda}}\subset\Omega_V$)
and $\mathcal{D}\subset\Omega_V \subset X$.  Recall the superscript $k$ denotes the $k$-th component of a $d$-dimensional vector.

Now we use an estimate similar to (3.16) from \cite[Lemma 3.9]{GieHaf15}: recall that $\hat{V}\in W^{\tau_2}_2(\Omega_V)$. Then from Corollary \ref{cor:smoothnessK1K2} we have
\begin{equation*}
||(\nabla \hat{V})^k||_{L^\infty(\Omega_V)} \le ||\hat{V}||_{C^1(\Omega_V)} \le C||\hat{V}||_{\mathcal{H}_{K^2}}.
\end{equation*}
Recall that $\hat{V}$ is the norm-minimal generalised interpolant to $V_{\mathbf{z},\lambda}$ in $\mathcal{H}_{K^2}$ (since $V_{\mathbf{z},\lambda}$ satisfies \eqref{eqn:Vzlambdaorbder}), and  $\mathcal{H}_{K^2}$ is norm-equivalent to $W^{\tau_2}_2(\Omega_V)$. Then $||\hat{V}||_{W^{\tau_2}_2(\Omega_V)} \le C ||{V}_{\mathbf{z},\lambda}||_{W^{\tau_2}_2(\Omega_V)}$.

In addition, Lemma \ref{lem:Vzlambdabounded} shows that
\begin{equation*}
||{V} - {V}_{\mathbf{z},\lambda}||_{W^{\tau_2}_2(\Omega_V)}  \le
C||{V} - {V}_{\mathbf{z},\lambda}||_{W^{\tau_2}_\infty(\Omega_V)}
\le  C||{V} - {V}_{\mathbf{z},\lambda}||_{C^{\nu_2}(\Omega_V)} \le C\varepsilon,
\end{equation*}
and so
 $ ||V_{\mathbf{z},\lambda} ||_{W^{\tau_2}_2(\Omega_V)} \le C ||  V ||_{W^{\tau_2}_2(\Omega_V)}$ for sufficiently small $||\mathbf{w}||_{\mathbb{R}^m}$, $h_{\mathbf{x}}$ with probability $1-\delta$.
Then it follows that
%\begin{eqnarray}					\label{eqn:Vestimate}
%\langle \nabla\hat{V}(x), f^*(x)\rangle_{\mathbb{R}^d} + p(x) & \le & C_1h_{\mathbf{q}}^{k_2-\frac{1}{2}} ||{V}||_{W^{\tau_2}_2(\Omega_V)} \nonumber\\
% && + \,C_2\max_k||f^k_{\mathbf{z},\lambda} - f^{*,k}||_{L^\infty(X)} ||  V ||_{W^{\tau_2}_2(\Omega_V)}.
%\end{eqnarray}
\begin{equation}					\label{eqn:Vestimate}
\langle \nabla\hat{V}(x), f^*(x)\rangle_{\mathbb{R}^d} + p(x)  \le  C_1||{V}||_{W^{\tau_2}_2(\Omega_V)} \left(
h_{\mathbf{q}}^{k_2-\frac{1}{2}}
 + \max_k||f^k_{\mathbf{z},\lambda} - f^{*,k}||_{L^\infty(X)} ||  \right).
\end{equation}
We may similarly show
\begin{equation}
\langle \nabla\hat{T}(x), f^*(x)\rangle_{\mathbb{R}^d} + c  \le   C_1 ||{T}||_{W^{\tau_2}_2(\Omega_T)}
\left(h_{\mathbf{q}}^{k_2-\frac{1}{2}}
 + \max_k||f^k_{\mathbf{z},\lambda} - f^{*,k}||_{L^\infty(X)}
\right).	\label{eqn:Testimate1}
\end{equation}

Furthermore, we can directly apply \eqref{eqn:ests2gamma} from Theorem \ref{thm:GieWen} to obtain
\begin{equation}				\label{eqn:Testimate2}
||\hat{T} - T||_{L^\infty(\Gamma)}  \le  Ch_{\tilde{\mathbf{q}}}^{k_2+\frac{1}{2}}||T||_{W_2^{\tau_2}(\Omega_T)}	
\end{equation}

Combining \eqref{eqn:festimate} with \eqref{eqn:Vestimate}--\eqref{eqn:Testimate2} proves Theorem \ref{thm:mainresult}.\qquad$\square$
%, where $\Omega = \tilde\Omega_V$ in the first part and $\tilde\Omega = \tilde\Omega_T$ in the second part.

%It remains to show that $||f^k_{\mathbf{z},\lambda} - f^{*,k}||_{L^\infty(X)}$ converges to zero as the data $\mathbf{z}:=(x_i, y_i)_{i=1}^m$ becomes  more dense, i.e. the quantities $h_{\mathbf{x}}$ and $\mathbf{w}$ tend to zero. This is the subject of \S\ref{sec:errorf}.

\begin{remark}
Furthermore, as $||f^k_{\mathbf{z},\lambda} - f^{*,k}||_{L^\infty(X)}$ converges to zero, we can shrink the ball $B_\varepsilon(\overline{x})$ (and therefore also $\tilde\Gamma$) towards $\overline{x}$. The domains $\Omega_V$ and $\Omega_T$ will converge towards $\tilde\Omega_V$ and $\tilde\Omega_T$ respectively, and therefore $V_{\mathbf{z},\lambda}$ and $T_{\mathbf{z},\lambda}$ will converge to $V$ and $T$ respectively. However, we do not give estimates for how fast $h_{\mathbf{x}}$ and $\mathbf{w}$ would need to converge to zero relative to $\varepsilon$.
\end{remark}

\section{Acknowledgements}
%The third and fourth authors gratefully acknowledge support from the UK Engineering and Physical Sciences Research Council (EPSRC).
B. Hamzi was supported by a Marie Curie Fellowship Grant Number 112C006. M. Rasmussen was supported by an EPSRC Career Acceleration Fellowship EP/I004165/1 and K.N. Webster was supported by the EPSRC Grant EP/L00187X/1 and
a Marie Sk\l odowska-Curie Individual Fellowship Grant Number 660616.

We would also like to thank Holger Wendland for drawing our attention to the reference for Theorem \ref{thm:WenRie}.

\end{document}